\documentclass[a4paper,11pt]{scrartcl}

\RequirePackage[OT1]{fontenc} \RequirePackage{amsthm}
\RequirePackage[cmex10]{amsmath} 
\RequirePackage{amsfonts} \RequirePackage{amssymb} \RequirePackage{a4wide}
\RequirePackage{graphicx}
\RequirePackage{xcolor}
\RequirePackage[normalem]{ulem}

\usepackage[authoryear]{natbib}

\numberwithin{equation}{section}
\RequirePackage{a4wide} 

\theoremstyle{plain}
\newtheorem{definition}{Definition}[section]
\newtheorem{theorem}{Theorem}[section]
\newtheorem{corollary}{Corollary}[section]
\newtheorem{example}{Example}[section]
\newtheorem{remark}{Remark}[section]
\newtheorem{lemma}{Lemma}[section]

\def\@bysame#1{\vrule height 1.5pt depth -1pt width 3em \hskip
	0.5em\relax}

\newcommand{\N}{ \mathbb{N} }

\newcommand{\R}{ \mathbb{R} }

\newcommand{\wh}[1]{ \widehat{ #1 } }
\newcommand{\wt}[1]{ \widetilde{ #1 } }

\newcommand{\calC}{\mathcal{C}}
\newcommand{\calD}{\mathcal{D}}

\newcommand{\calF}{\mathcal{F}}

\newcommand{\eins}{\mathbf{1}}

\newcommand{\Var}{{\mbox{Var\,}}}

\newcommand{\chg}[2]{#2}

\begin{document}

\begin{center}
	\begin{minipage}{.8\textwidth}
		\centering 
		\LARGE Jackknife variance estimation for general two-sample statistics and applications to common mean estimators under ordered variances \\[0.5cm]

		\normalsize
		\textsc{Ansgar Steland and Yuan-Tsung Chang}\\[0.1cm]
		Institute of Statistics\\
		RWTH Aachen University\\
		Aachen, Germany\\
		Email: \verb+steland@stochastik.rwth-aachen.de+ \\ {\em and} \\
		
		Department of Social Information\\ 
		Mejiro University \\ 
		4-31-1 Nakaochiai, Shinjuku--ku \\ 
		Tokyo 161--8539, Japan 
		
	\end{minipage}
\end{center}


\begin{abstract}
	We study the jackknife variance estimator for a general class of two-sample statistics. As a concrete application we consider
  samples with a common mean but possibly different, ordered variances as arising in various fields such as interlaboratory experiments, field studies or the analysis of sensor data.  Estimators for the common mean under ordered variances typically employ random weights, which depend on the sample means and the unbiased variance estimators. They take different forms when the sample estimators are in agreement with the order constraints or not, which complicates even basic analyses such as estimating their variance. We propose to use the jackknife, whose consistency is established for general smooth two--sample statistics induced by continuously G\^ateux or Fr\'echet differentiable functionals, and, more generally, asymptotically linear two--sample statistics, allowing us to study a large class of common mean estimators. Further, it is shown that the common mean estimators under consideration satisfy a central limit theorem (CLT). We investigate the accuracy of the resulting confidence intervals by simulations and illustrate the approach by analyzing several data sets. 
\end{abstract}



\textit{Keywords:} Common mean; Data science; Central limit theorem; G\^ateaux derivative; Fr\'echet differentiability;  Graybill--Deal estimator; Jackknife;  Order constraint; Resampling

\section{Introduction}

We study the jackknife variance estimation methodology for a wide class of two-sample statistics including asymptotically linear  statistics and statistics induced by differentiable two-sample statistical functionals, thus extending the well studied one-sample results, see  \cite{EfronStein1981}, \cite{ShaoWu1998} and \cite{Steland2015} and the discussion below. Comparison of two samples by some statistic is a classical statistical design and also widely applied to analyze massive data arising in data science, e.g. when exploring such data by analyzing subsets. Our specific motivation comes from the following classical common mean estimation problem: 
In many applications several samples of measurements with common mean but typically with different degrees of uncertainties (in the sense of variances) are drawn. \chg{}{This setting generally arises when using competing measurement systems of different quality or if the factor variable defining the samples affects the dispersion. For example, when checking the octane level in gasoline at pump stations, inspectors use cheap handheld devices to collect many low precision measurements and send only a few samples to government laboratories for detailed analyses of high precision. In both cases the same mean octane level is measured, but the variance and the shape of the distribution may differ. The issue of samples with common mean but heterogeneous, ordered variances also arises in big data applications, for example (e.g.) when processing data from image sensors, see  \cite{Degerli} and \cite{Lin}, or accelerometors, \cite{Cemer}, as used in smartphones or specialized measurement systems. Here the thermo-mechanical (or Brownian) noise represents a major source of noise and depends on  temperature. Therefore, in the presence of a constant signal, samples taken under different conditions exhibit different variances and the order constraint is related to temperature.} It is worth mentioning that the general statistical problem how to combine estimators from different samples dates back to the works of \cite{Fisher1932}, \cite{Tippet1931} and \cite{Cochran1937}, cf. the discussion
given in \cite{KellerOlkin2004}.  In the present article, we study the classical problem to estimate the common mean in the presence of ordered variances and propose to use the a two-sample jackknife variance estimator to assess the uncertainty. 

The jackknife is easy to use and feasible for big data problems, since its computational costs are substantially lower than other techniques such as the bootstrap. Indeed, our simulations indicate that it also provides substantially higher accuracy of confidence intervals than the bootstrap for the common mean estimation problem. We therefore extend the jackknife methodology for smooth statistical functionals to a general two-sample framework and establish a new result, which holds as long the statistic of interest can be approximated by a linear statistic. This result goes beyong the case of smooth statistics induced by continuously differentiable functionals and allows us to treat a large class of common mean estimators. 

Since the jackknife has not yet been studied for two--sample settings, we establish its consistency and asymptotic unbiasedness for possibly nonlinear but asymptotically linear two--sample statistics. We introduce a specific new jackknife variance estimator for two samples with possibly unequal sample sizes $ n_1 $ and $ n_2 $, which is based on $ n_1 + n_2 $ leave-one-out replicates. In addition, for equal sample sizes we study an alternative procedure which generates replicates by leaving out pairs of observations. Both jackknife estimators are shown to be weakly consistent and asymptotically unbiased. Those general results allow us to show that the jackknife consistently estimates the variance of a large class of common mean estimators including many of those proposed in the literature. \chg{}{They are, however, also interesting in their own right. Firstly, because we provide conditions which are easier to verify in cases where the statistic of interest is not induced by a smooth statistical functional, e.g. due to discontinuities as arising in the common mean estimation problem. Secondly, since we provide a proof using elementary arguments avoiding the calculus of differentiable statistical functionals. Lastly, our approach shows that the pseudo-values are consistent estimates of the summands of the asymptotically equivalent linear statistic. In addition to these results, we also extend the known consistency results for continuously G\^ateaux- and Fr\'echet-differentiable statistical functionals addressing one-sample settings to two-sample settings resulting in a comprehensive treatment of the two-sample jackknife methodology.}

For the common mean estimation problem, which we studied in depth as a non-trivial application,  several common mean estimators have been discussed in the literature.
For unequal variances the Graybill-Deal (GD) estimator may be used, which weights the sample averages with the inverse unbiased variance estimates. If, however, an order constraint on the variances is imposed, several estimators have been proposed which dominate the GD estimator, see especially \cite{Nair1982}, \cite{Elfessi1992}, \cite{ChangEtAl2008} and \cite{ChangEtAl2012}. Those estimators are given by convex combinations of the sample means with weights additionally depending on whether the ordering of sample variances is in agreement with the order constraint on the variances.

When random weights are used, even the calculation of the variance of such a common mean estimator is a concern and has been only studied under the assumption of Gaussian samples for certain special cases such as the GD estimator. As a way out, we propose to employ the jackknife variance estimator of \cite{Quenouille1949} and \cite{Tukey1958}, which basically calculates the sample variance of a pseudo sample obtained by leave-one-out replicates of the statistic of interest. It has been extensively studied for the case of one sample problems by \cite{EfronStein1981} and \cite{ShaoWu1998}, the latter for asymptotically linear statistics, and recently by \cite{Steland2015} for the case of vertically weighted sample averages. Compared to other approaches the jackknife variance estimator has the advantage not to give underbiased estimates, cf. \cite[p.~42]{Efron1982} and \cite{EfronStein1981}. It also made its way in recent textbooks devoted to computer age statistics and data science, see \cite{EfronHastie2016}.

A further issue studied in this paper is the asymptotic distribution of common mean estimators. We show that common mean estimators using random weights are asymptotically normal under fairly weak conditions, which are satisfied by those estimators studied in the literature. Combining this result with the jackknife variance estimators allows us to construct asymptotic confidence intervals and to test statistical hypotheses about the common mean. \chg{}{In \cite{Steland2017} that approach was applied to real data from photovoltaics, compared with other methods and investigated by data-driven simulations using distributions arising in that field. It was found that confidence intervals based on the proposed methodology have notably higher accuracy in terms of the coverage probability, thus providing a real world example where the approach improves upon existing ones. In this paper, we broaden these data-driven results by a simulations study investigating distributions with different tail behavior.}

The organization of the paper is as follows. Section~\ref{sec: jackknife} studies the jackknife variance estimator. Two-sample statistics induced by G\^ateaux- and Fr\'echet-differentiable functionals are studied as well as asymptotically linear two--sample statistics. In Section~\ref{sec: Estimation} we introduce and review the common mean estimation problem for unequal and  ordered variances for two samples and discusses related results from the literature. Section~\ref{sec: jackknife common mean} presents the results about the jackknife variance estimation for common mean estimators with random weights and provides a central limit theorem. Lastly, Section~\ref{sec: data_analysis} presents the simulations and analyzes three data sets from physics, technology and social information, in order to illustrate the approach.

\section{The jackknife for two-sample statistics}
\label{sec: jackknife}

The Quenouille--Tukey jackknife, see \cite{Quenouille1949},
\cite{Tukey1958} and \cite{Miller1974}, is a simple and effective resampling technique for bias and variance estimation, see also the monograph \cite{Efron1982}. For a large class of one sample statistics the consistency of the jackknife variance estimator has been studied in depth by \cite{ShaoWu1998} and recently in \cite{Steland2015}. \cite{LeeY1991} studied jackknife variance estimation for a one-way random effects model by simulations. In the present section, we extend the jackknife to quite general two--sample statistics. To the best of our knowledge, the results of this section are new and extend the existing theoretical investigations. 

We shall first discuss the case of {\em asymptotically linear} two-sample statistics considering the cases of unequal and equal sample sizes separately, as it turns out that for equal sample sizes one may define a simpler jackknife variance estimator. Then we study the case of two-sample statistics induced by a two-sample differentiable statistical functional extending the known results for the one-sample setting.

Let us consider a two-sample \chg{test}{} statistic $ T_n $, i.e. a statistic $ T_n = T_n( X_{11}, \dots, X_{1n_1}, X_{21}, \dots, X_{2n_2} ) $ which is symmetric in its first $ n_1 $ arguments as well as in its last $ n_2 $ arguments; $ n = n_1 + n_2 $ is the total sample size. We assume that $ T_n $ has the following
property: For all sample sizes $ n_1, n_2 \in \N $ with $ n_i/n \to \lambda_i $, $ i = 1, 2 $, as $ \min(n_1,n_2) \to \infty $, and all independent i.i.d. samples $ X_{ij} \sim F_i $, $ j = 1, \dots, n_i $, $ i = 1, 2 $, we have 
\begin{equation}
\label{TwoSampleAss}
T_n = L_n+ R_n, \qquad E(L_n) = 0, 
\end{equation}
with a linear statistic
\begin{equation}
\label{LinearTerm}
L_n = \frac1n \left[ \sum_{i=1}^{n_1} h_1(X_i ) + \sum_{i=n_1+1}^{n} h_2(X_i)\right],
\end{equation}
for two kernel functions $ h_1, h_2 $ with 
\begin{equation}
\label{FourthMomentTwoSampleStat}
\int h_i^4 d F_i < \infty, \qquad i = 1,2,
\end{equation}
\chg{}{and a remainder term $ R_n $ satisfying $ n E(R_n^2) = o(1) $, as $ n \to \infty $.
	Here} and in what follows 
\[ 
(X_1, \dots, X_n)' = (X_{11}, \dots, X_{1n_1}, X_{21}, \dots, X_{2n_2})'.
\] 
\chg{}{Recall that a statistics $ T_n $ attaining such a decomposition is called {\em asymptotically linear}.}

Let us recalling the following definitions and facts.
Suppose that $ U_n $ is a statistic which satisfies, for some parameter $ \theta $, the central limit theorem  
\[ \sqrt{n} (U_n - \theta) \stackrel{d}{\to} N( 0, \sigma^2(U) ), \] as $ n \to \infty $, with $ \lim_{n \to \infty} n \Var( U_n ) = \sigma^2(U)  \in (0, \infty) $. Here $ \stackrel{d}{\to} $ denotes convergence in distribution. Then the constant $ \sigma^2(U) $ is called {\em asymptotic variance of $U_n$}. 
By (\ref{TwoSampleAss}), $ T_n $ inherits its asymptotic variance denoted $ \sigma^2(T) $  from $ L_n $ and thus we obtain
\begin{align} \nonumber
\sigma^2(T) &= \lim_{n \to \infty} \Var( \sqrt{n} L_n )  \\
& = \lim_{n \to \infty} n \left( \frac{n_1}{n} \right)^2 \frac{ \Var( h_1(X_1) ) }{n_1}
+ n \left( \frac{n_2}{n} \right)^2 \frac{ \Var( h_2(X_{n_1+1}) ) }{n_2} \\
\label{FormulaAsVarLinearization}
& = \lambda_1 \tau_1^2 + \lambda_2 \tau_2^2,
\end{align}
where \[\tau_1^2 = \Var( h_1(X_{11}) ) \quad \text{and} \quad  \tau_2^2 = \Var( h_2(X_{21}) ). \]

\subsection{Unequal samples sizes}

First we discuss the general case of unequal sample sizes. \chg{}{Let us make the dependence on the observations explicit and write
	\[  
	T_n = T_n( X_1, \dots, X_{n_1}, X_{n_1+1}, \dots,  X_n).
	\] 
	The leave-one-out statistic obtained when omitting the $i$th observation is denoted by
	\[
	T_{n,-i} =  T_{n-1}( X_1, \dots, X_{i-1}, X_{i+1}, \dots, X_n ),
	\] 
	for $ i = 1, \dots, n$.} Define the leave-one-out pseudo values
\begin{equation}
\label{JackknifeLeaveOneOut}
\wh{\xi}_i = n T_n - (n-1) T_{n,-i}, \qquad i = 1, \dots, n.
\end{equation}
The jackknife leave-one-out variance estimator for $ \sigma^2(T) $ is then
defined by
\begin{equation}
\label{JackknifeEst2Sample}
\wh{\sigma}^2_n(T) = \frac{n_1}{n} \wh{\tau}^2_{1} + \frac{n_2}{n} \wh{\tau}^2_{2},
\end{equation}
where
\[
\wh{\tau}^2_{1} = \frac{1}{n_1-1} \sum_{j=1}^{n_1} ( \wh{\xi}_j - \overline{\wh{\xi}}_{1:n_1} )^2
\]
and
\[
\wh{\tau}^2_{2} = \frac{1}{n_2-1} \sum_{j=n_1+1}^{n_2} ( \wh{\xi}_j - \overline{\wh{\xi}}_{(n_1+1):n} )^2
\]
with $ \overline{\wh{\xi}}_{a:b} = \frac{1}{b-a+1} \sum_{j=a}^b \wh{\xi}_j $. 

The associated
jackknife variance estimator of $ \Var( T_n ) $ is then given by
\[
\wh{\Var}( T_n ) = \frac{ \wh{\sigma}_n^2(T) }{ n },
\]
see also \cite{EfronStein1981}, \cite{Efron1982} and the discussion given in Remark~\ref{Remark_Norming}.

\begin{theorem}
	\label{TwoSampleJack}
	Assume that $T_n $ satisfies (\ref{TwoSampleAss}), (\ref{LinearTerm}) and
	(\ref{FourthMomentTwoSampleStat}). If, in addition, the remainder term satisfies
	\begin{equation}
	\label{CrucialCondRemainder1}
	n E( R_n^2 ) = o(1),
	\end{equation}
	\begin{equation}
	\label{CrucialCondRemainder}
	n^2 E( R_n - R_{n-1} )^2 = o(1),
	\end{equation}
	as $ n \to \infty $, then the following assertions hold.
	\begin{itemize}
		\item[(i)] For each $ 1 \le i \le n $: \[ E(\wh{\xi}_i - \xi_i)^2 = o(1), \]  where
		$ \xi_i =  h_1(X_i) $, if $ 1\le i \le n_1 $, and $ \xi_i = h_2(X_i)  $, if  $ n_1+1\le i \le n $.
		\item[(ii)] $ \wh{\sigma}^2_n(T) $ is a consistent and asymptotically unbiased estimator for the asymptotic variance $ \sigma^2(T) $ of $ T_n $, i.e.
		\[
		\left| \frac{ \wh{\sigma}^2_n(T) }{ \sigma^2(T) } - 1 \right| \to 0,
		\] 
		in probability, as $ \min(n_1,n_2)  \to \infty $ and 
		\[
		\left| \frac{ E\wh{\sigma}^2_n(T)  }{ \sigma^2(T) } - 1 \right| \to 0,
		\]
		as $ \min(n_1,n_2)  \to \infty $.   The associated
		jackknife variance estimator of $ \Var( T_n ) $ shares the above consistency properties.
	\end{itemize}
\end{theorem}

\begin{remark}
	Observe that examples for (\ref{CrucialCondRemainder1}) and (\ref{CrucialCondRemainder}) to hold are $  h_i(x) = x - \mu_i $ (arithmetic means) and $ h_i(x) = (x-\mu_i)^2 - \sigma_i^2 $ (sample variances, as verified in the appendix). The conditions on the second moment of the remainder term in (\ref{CrucialCondRemainder1}) and (\ref{CrucialCondRemainder}) can be interpreted as measures of smoothness of $ T_n $. They have been also employed by \cite{ShaoWu1998}, cf. their Theorem~1, to study jackknife variance estimation for one sample statistics, but our proof is quite different from the methods of proof used there. Especially, we show that, by virtue of condition (\ref{CrucialCondRemainder}), the summands of the asymptotic linearization, $ L_n $, i.e. the random variables $ h_1(X_i) $, $ i = 1, \dots, n_1 $, and $ h_2(X_i) $, $ i = n_1 + 1 ,\dots, n_1+n_2 $, can be estimated consistently. To the best of our knowledge, this interesting and useful result has not yet been established in the literature.
\end{remark}

\begin{proof} We provide a direct proof. \chg{}{As a preparation, observe the following facts about the statistic $ T_n $  and its leave-one-out pseudo values:} When omitting the $i$th observation, say $ i \le n_1 $, and calculating the statistic from the resulting sample of size $n-1$, we \chg{obtain the leave-one-out statistic}{have}
	\[
	T_{n,-i} =  T_{n-1}( X_1, \dots, X_{i-1}, X_{i+1}, \dots, X_n )
	\] 
	where $ ( X_1, \dots, X_{i-1}, X_{i+1}, \dots, X_n ) \stackrel{d}{=} 
	( X_1', \dots, X_{n-1}') $ whenever the first $ n_1-1 $ observations 
	$ X_1', \dots, X_{n_1-1}' $ are i.i.d. with common d.f. $ F_1 $ and the remaining $ X_i' $s are i.i.d. with d.f. $ F_2 $. Hence,
	\[
	T_{n,-i} = L_{n,-i} + R_{n,-i} \stackrel{d}{=} L_{n-1} + R_{n-1},
	\]
	where $ L_{n,-i} + R_{n,-i} $ denotes the decomposition of $ T_{n,-i} $ into a linear statistic and a remainder term and  $ L_{n-1} + R_{n-1} $ is the decomposition  when applying the statistic to the sample $ X_1, \dots, X_{n_1-1}, X_{n_1+1}, \dots, X_{n} $, i.e. when the {\em last} observation of the first sample is omitted. In the same vain, when omitting an arbitrary observation of second sample, the resulting decomposition of the statistic is equal in distribution to the decomposition obtained when omitting the last observation of the second sample.
	In particular, the second moment of the remainder term corresponding to $ T_{n,-i} $ does not depend on $ i $. By (\ref{TwoSampleAss}), we can write $ T_n $ as 
	\[
	T_n = \frac{n_1}{n} \frac{1}{n_1} \sum_{j=1}^{n_1} h_1(X_j) + \frac{n_2}{n} \frac{1}{n_2} \sum_{j=n_1+1}^{n} h_2(X_j) + R_n. 
	\]
	To show the validity of the jackknife variance estimator, we shall focus on the first sample and put $ \xi_i = h_1(X_i) $, $ i = 1, \dots, n_1 $.

	Let $ R_{n,-i} $ be the remainder term arising in that decomposition when omitting the $i$th observation. This means,
	$ T_{n,-i} = L_{n,-i} + R_{n,-i} $ with $
	L_{n,-i} = \frac{n_1-1}{n-1} \frac{1}{n_1-1} \sum_{j=1, j \not= i}^{n_1} h_1(X_j) + \frac{n_2}{n-1} \frac{1}{n_2} \sum_{j=n_1+1}^{n-1} h_2(X_j) $, if $ 1 \le i \le n_1 $ and
	$ L_{n,-i} = \frac{n_1}{n-1} \frac{1}{n_1} \sum_{j=1}^{n_1} h_1(X_j) + \frac{n_2-1}{n-1} \frac{1}{n_2-1} \sum_{j=n_1+1, j \not=i}^{n} h_2(X_j)  $, if $ n_1+1 \le i \le n $, and
	$ R_{n,-i} = T_{n,-i} - L_{n,-i} $ for $ i = 1, \dots, n$. By assumption, $
	n^2E(R_{n,-i})^2 = o(1) $ holds for all $i = 1, \dots, n $.
	To show the validity of the jackknife variance estimator, we shall focus on the first sample and put $ \xi_i = h_1(X_i) $, $ i = 1, \dots, n_1 $.
	By definition (\ref{JackknifeLeaveOneOut}), the leave-one-out pseudo values are then given by
	\[
	\wh{\xi}_i = n T_n - (n-1) T_{n,-i} = h_1(X_i) + A_{ni}, \quad A_{ni} = nR_n - (n-1) \wt{R}_{n-1,-i}
	\]
	with
	\[
	\wt{R}_{n-1,-i} \stackrel{d}{=} R_{n-1,-1}, \qquad 1 \le i \le n_1
	\]
	(and
	$
	\wt{R}_{n-1,-i} \stackrel{d}{=} R_{n-1,n-1} =: R_{n-1} $ for $  n_1+1 \le i \le n.
	$). This means
	\begin{equation}
	\label{ApproxXiPr}
	\wh{\xi}_i = \xi_i + A_{ni}, \qquad  \max_{1 \le i \le n_1} n E(A_{ni}^2) = o(1).
	\end{equation}
	The same arguments using $ R_{n,-i} \stackrel{d}{=} R_{n,-n} $ show that
	(\ref{ApproxXiPr}) holds for all $ i = 1, \dots, n $. Since
	$ R_{n,-i} \stackrel{d}{=} R_{n-1} $ as well, we may
	conclude that
	\begin{align*}
	\max_{1 \le i \le n} E( A_{ni}^2 ) & =   \max_{1 \le i \le n}  E( n R_n - (n-1) R_{n,-i} )^2 \\
	& =   \max_{1 \le i \le n}  E( n(R_n - R_{n,-i} )^2 + R_{n,-i} )^2 \\
	& =    n^2 E( R_n - R_{n-1} )^2 + E( R_{n-1}^2 ) 
	+ 2 n E( (R_n-R_{n-1}) R_{n-1} ) \\
	& = o(1),
	\end{align*}
	where the last term is estimated by Cauchy-Schwarz inequality:
	\[
	n | E( (R_n - R_{n-1})  R_{n-1} ) | \le \sqrt{ n^2 E( (R_n-R_{n-1})^2) } \sqrt{ E( R_{n-1}^2 ) } = o(1).
	\]
	Clearly, $ \tau_1^2 $ can be estimated consistently by the sample variance
	of $ \xi_1, \dots, \xi_{n_1} $. We have to show that we may replace the $ \xi_i $'s by the $ \wh{\xi}_i $'s. As shown above, $ \wh{\xi}_i = \xi_i + A_n $ where $ E(A_n^2) = o(1) $. Hence
	\[
	E|\wh{\xi}_i^2 - \xi_i^2| = E | (\xi_i + A_n)^2 - \xi_i^2 | = 2 E| \xi_i A_n | + E (A_n^2)
	\le 2 \sqrt{ E( \xi_1^2 ) } \sqrt{ E (A_n^2) } + o(1) = o(1), 
	\]
	for $ i = 1, \dots, n $,
	such that the $ L_1 $--convergence of the sample moment of the squares and of the squared sample moments follows from Lemma~\ref{AppLemma1}. We can therefore conclude that
	\[
	E \left| \frac{1}{n_1} \sum_{j=1}^{n_1} ( \wh{\xi}_j - \overline{\wh{\xi}}_{1:n_1} )^2 
	- \frac{1}{n_1} \sum_{j=1}^{n_1} ( \xi_j - \overline{\xi}_{1:n_1} )^2 \right| =o(1),
	\]
	as $ n_1  \to \infty $. Combining this with the fact that the sample variance 
	$ \frac{1}{n_1} \sum_{j=1}^{n_1} \xi_j^2 - ( \frac{1}{n_1} \sum_{j=1}^{n_1} \xi_j )^2 $ of the $ \xi_j $'s is $ L_1 $--consistent if $ E| \xi_1 |^4 < \infty $, we obtain the $ L_1 $--consistency of the jackknife variance estimator for $ \tau_1^2 $, 
	\[
	\wh{\tau}^2_{1} = \frac{1}{n_1-1} \sum_{j=1}^{n_1} ( \wh{\xi}_j - \overline{\wh{\xi}}_{1:n_1} )^2, 
	\]
	which implies the (weak) consistency in the sense that 
	\begin{equation}
	\label{ConsistencySampleOne}
	\left| \frac{ \wh{\tau}^2_{1}  }{ \tau_1^2 } - 1 \right| = \frac{ | \wh{\tau}^2_{1} - \tau_1^2 | }{ \tau_1^2 } \stackrel{P}{\to} 1,
	\end{equation}
	as $ n_1 \to \infty $, and also yields the asymptotic unbiasedness by taking the expectation of the left hand side of (\ref{ConsistencySampleOne}). 
	
	In the same vain, the jackknife variance estimator
	$ \wh{\tau}^2_{2} $
	for $ \tau_2^2 $ is $ L_1 $--consistent, weakly consistent and asymptotically unbiased. Consequently, we may estimate the asymptotic variance of $ T_n $ by the jackknife variance estimator
	\[
	\wh{\sigma}^2_n(T) = \frac{n_1}{n} \wh{\tau}^2_{1} + \frac{n_2}{n} \wh{\tau}^2_{2}.
	\]
	and obtain
	\[
	E| \wh{\sigma}^2_n(T)  - \sigma^2(T) | = o(1),
	\]
	as $ \min(n_1,n_2)  \to \infty $,  
	\[
	\left| \frac{ \wh{\sigma}^2_n(T)  }{ \sigma^2(T) } - 1 \right| \to 0,
	\] 
	in probability, as $ \min(n_1,n_2)  \to \infty $ and the asymptotic unbiasedness,
	\[
	\left| \frac{ E\wh{\sigma}^2_n(T)  }{ \sigma^2(T) } - 1 \right| \to 0,
	\]
	as $ \min(n_1,n_2)  \to \infty $. 
\end{proof}

\subsection{Equal sample sizes $ N = n_1 = n_2 $}

For equal sample sizes $ N = n_1 = n_2 $, such that $ n = 2N $, one may, of course, apply the jackknife variance estimator discussed above. However, for such balanced two-sample designs one can propose a simpler jackknife variance estimator when viewing $ T_n $ as a statistic of the $N$ pairs 
\[
Z_j = ( X_{1j}, X_{2j} ), \qquad j = 1, \dots, N.
\]
Thus we write
\[
T_n = T_N( Z_1, \dots, Z_N ) 
\]
to make the dependence on the $N$ pairs of random vectors explicit in our notation. Let
\begin{align*}
T_{N,-i} &= T_{N-1}( Z_1, \dots, Z_{i-1}, Z_{i+1}, \dots, Z_N ) \\
& = T_{2N-2}( X_{11}, \dots, X_{1,i-1}, X_{1,i+1}, \dots, X_{1N}, X_{21}, \dots, X_{2,i-1}, X_{2,i+1}, \dots, X_{2N} )
\end{align*}
denote the {\em leave-one-pair-out} statistics and define the 
the {\em leave-one-pair-out} pseudo values
\[
\wh{\xi}_i = N T_N( Z_1, \dots, Z_N ) - (N-1) T_{N,-i}.
\]
The jackknife variance estimator for the asymptotic variance $ \sigma^2(T) $ is now given by
\begin{equation}
\label{JackEqualSizesAsVar}
\wh{\sigma}_N^2(T) 
= \frac{1}{N-1} \sum_{i=1}^N ( \wh{\xi}_i - \overline{\wh{\xi}}_N )^2
= (N-1) \sum_{i=1}^N ( T_{N,-i} - \overline{T_{N,-\bullet}} )^2,
\end{equation}
and the corresponding jackknife variance estimator of $ \Var( T_N ) $ is 
\begin{equation}
\label{JackEqualSizesVar}
\wh{\Var}(T_N) = \frac{\wh{\sigma}_N^2(T)}{ N } = 
\frac{N-1}{N} \sum_{i=1}^N ( T_{N,-i} - \overline{T_{N,-\bullet}} )^2.
\end{equation}

\begin{theorem} For equal sample sizes the jackknife estimators
	(\ref{JackEqualSizesAsVar}) and (\ref{JackEqualSizesVar}) are consistent and
	asymptotically unbiased.
\end{theorem}

\begin{proof} The proof goes along the lines of the proof of Theorem~\ref{TwoSampleJack}. Therefore, we indicate the main changes and simplifications. Observe that for equal sample sizes the representation of  $ T_n $ in terms of a linear statistic and a negligible remainder term simplifies to
	\[
	T_N = \frac{1}{N} \sum_{j=1}^N \frac{1}{2} [h(X_{1j}) + h( X_{2j} )] + R_N = \frac{1}{N} \sum_{j=1}^N h(Z_j) + R_N, \quad N E(R_N^2) = o(1),
	\]
	with kernel function
	\[
	\xi(z) = \frac{1}{2} [h_1( z_1 ) + h_2(z_2)], \qquad z = (z_1, z_2)' \in \R^2,
	\]
	not depending on $n$. The asymptotic variance of $ T_N $ is now given by
	\[
	\sigma^2(T) = \lim_{N \to \infty} \Var( \sqrt{N} T_N ) = \Var( \xi( Z_1 ) )
	\]
	It follows that the leave-one-out pseudo values satisfy
	\[
	\wh{\xi}_i = N T_N - (N-1) T_{N,-i} = \xi(Z_i) + A_N
	\]
	where arguments as detailed in the proof of Theorem~\ref{TwoSampleJack} show that the remainder term $ A_N $ satisfies $ N E( A_N^2) = o(1) $ and $ N^2 E( R_N - R_{N-1})^2 = o(1) $, as $ N \to \infty $. Therefore, by repeating the arguments given there for the general case of unequal sample sizes, we see that a  $L_1$--consistent, weakly consistent and asymptotically unbiased estimator of  $ \sigma^2(T) $ is given by the sample variance of the pseudo values $ \wh{\xi}_1, \dots, \wh{\xi}_N $ and hence the consistency of 
	\[
	\wh{\sigma}_N^2(T) = \frac{1}{N-1} \sum_{i=1}^N ( \wh{h}_i - \overline{\wh{h}}_N )^2
	= (N-1) \sum_{i=1}^N ( T_{N,-i} - \overline{T_{N,-\bullet}} )^2,
	\]
	follows, which coincides with the proposed jackknife estimator.
\end{proof}

\begin{remark}
	\label{Remark_Norming}
	In (\ref{JackEqualSizesAsVar}) the variance estimator leading to unbiased estimation for i.i.d. samples is used. Using the sample variance formula gives instead
	\begin{equation}
	\label{JackEqualSizesAsVar2}
	\widetilde{\sigma}_N^2(T) = \frac{ (N-1)^2 }{N} \sum_{i=1}^N ( T_{N,-i} - \overline{T_{N,-\bullet}} )^2
	\end{equation}
	and
	\begin{equation}
	\label{AsVarEqual2}
	\widetilde{\Var}(T_N) =  \frac{ \wt{\sigma}_N^2(T) }{ N } =  \left( \frac{N-1}{N} \right)^2 \sum_{i=1}^N ( T_{N,-i} - \overline{T_{N,-\bullet}} )^2.
	\end{equation}
	Observe that we may put any non-random factor $ f_N$ in front of the sum,  $ \sum_{i=1}^N ( T_{N,-i} - \overline{T_{N,-\bullet}} )^2 $, in	(\ref{JackEqualSizesVar}) which satisfies $ f_N \to 1 $, as $ N \to \infty $. The choice $ (N-1)/N $ is usually justified by the fact that then the jackknife variance estimator matches the formula used for the arithmetic mean, see e.g. \cite[p.~142]{EfronTib1993}. 
	The same comments apply to the formulas provided for unequal sample sizes.
\end{remark}

For equal sample sizes we may also directly refer to the results of \cite{ShaoWu1998} to obtain consistency of the delete-$d$ jackknife variance estimator. Let $ S_{N,r} $ be the collection of subsets of $ \{ 1, \dots, N \} $ which have size $r = N-d $. For $ s = \{ i_1, \dots, i_r \} \in S_{N,r} $ let $ T_N^{(s)} = T_N( \xi_{i_1}, \dots, \xi_{i_r} ) $. The delete-$d$ jackknife variance estimator is then defined by
\[
\wh{\sigma}_{J(d)}^2 = \frac{ r }{ d {N \choose d} } \sum_{s \in S_{N,r}} ( T_N^{(s)} - T_N)^2.
\]

\begin{corollary} Suppose that $ N E( R_N^2 ) = o(1) $ and $ d = d_N $ satisfies
	\[
	d/N \ge \varepsilon_0 \quad \text{for some $ \varepsilon_0 > 0 $} 
	\]
	and $ r = N-d \to \infty $. Then the delete-$d$ jackknife variance estimator $ \wh{\sigma}_{J(d)}^2 $ is consistent in the sense that $ N \wh{\sigma}_{J(d)}^2 - \sigma^2 = o_P(1) $ and asymptotically unbiased in the sense that $  N \wh{\sigma}_{J(d)}^2 - \sigma^2 = o(1) $, as $ N \to \infty $.
\end{corollary}

\begin{proof}
	We have the decomposition $ T_N = L_N + R_N $ where the linear statistic $ L_N $ can be written as $ L_N = \frac1N \sum_{i=1}^N \xi_i $ with $ \xi_i = \frac12 [ h_1(X_{1i}) + h_2(X_{2i}) ] $, $ i = 1, \dots, N $. The proof of Theorem~1 in \cite{ShaoWu1998} uses the representation $ N \wh{\sigma}_{J(d)}^2 = \frac{N r}{d {N \choose d}} \sum_s (L(s) - U_s)^2 $, where $ L(s) = \frac{1}{r} \sum_{i \in s} \xi_i  - \frac{1}{N} \sum_{i=1}^N \xi_i $ and $ U_s = R_N - R_{N,s} $ with $ R_{N,s} $ the remainder term associated to $ T_N^{(s)} $, and then only refers to the properties of $ L(s) $ and  $ U_s $ and does not refer to the original definition of $ T_N $. Since the linear term has the same form as in \cite{ShaoWu1998}, the proof directly carry overs and the stated conditions are sufficient to ensure \cite[(3.3)]{ShaoWu1998}, cf.  Corollary~1 therein.
\end{proof}

\subsection{Two sample G\^ateaux and Fr\'echet differentiable functionals}

Let us first study the case of statistics induced by sufficiently smooth statistical functionals. Let $ \calF $ be the (convex) set of distribution functions on $ \R $ and let $ \calD = \{ (G_1,G_2) -(H_1,H_2) : (G_1, G_2), (H_1,H_2) \in \calF^2 \} $ be the  linear space associated to $ \calF^2 = \calF \times \calF $. $ d \delta_x$, $ x \in \R $, denotes the Dirac measure in $ x $ with distribution function $ \delta_x(z) = \eins( x \le z ) $, $ z \in \R $. 

Denote by $ \wh{F}_{n_i}^{(i)}(x) = \frac{1}{n_i} \sum_{j=1}^{n_i} \eins( X_{ij} \le x ) $, $ x \in \R $, the empirical distribution function of the $i$th sample, $ i = 1, 2 $. A functional $ T: \calF \times \calF \to \R $ induces a two-sample statistic, namely $ T_n = T( \wh{F}_{n_1}^{(1)}, \wh{F}_{n_2}^{(2)} ) $. A special case of interest is when $ T $ is {\em additive} with respect the distribution functions, i.e. $ T_n = T_1(  \wh{F}_{n_1}^{(1)} ) + T_2( \wh{F}_{n_2}^{(2)} ) $ for two one-sample functionals $ T_1, T_2 : \calF \to \R $, which we shall call components in what follows.

We will first consider statistics induced by a G\^ateaux-differentiable statistical functional. Here, additional assumptions are required to obtain consistenc of the jackknife variance estimator.  Fr\'echet-differentiability is a stronger notion and implies the consistency of the jackknife.

\begin{definition}
	A two sample functional $ T : \calF \times \calF \to \R $ is called G\^ateaux differentiable at $ (G_1, G_2) \in \calF^2 $ with G\^ateaux derivative $ L_{(G_1,G_2)} : \calD \to \R $, if
	\[
	\lim_{t \to 0} \frac{ T( (G_1,G_2) + t (D_1 D_2) ) - T(G_1,G_2) - L_{(G_1,G_2)}( t (D_1, D_2) ) }{ t } = 0,
	\]
	holds for all $ (D_1, D_2) \in \calD $. $ T $ is called {\em continuously G\^ateaux differentiable} at $ (G_1, G_2) $, if
	for any sequence $ t_k \to 0 $ and for all sequences $ ( G_1^{(k)}, G_2^{(k)} )  $, $ k \ge 1 $, with
	$ \max_{i=1,2} \| G_i^{(k)} - G_i \|_\infty \to 0 $, as $ k \to \infty $, 
	\[
	\sup_{x,y \in \R} \left| 
	\frac{ T( (G_1^{(k)},G_2^{(k)}) + t_k ( \delta_x - G_1^{(k)}, \delta_y - G_2^{(k)} ) ) - T( G_1^{(k)}, G_2^{(k)} )  }{ t_k } 
	- L_{(G_1,G_2)}( \delta_x - G_1^{(k)}, \delta_y - G_2^{(k)} ) 
	\right|
	\] 
	converges to $0$, as $ k \to \infty $.
\end{definition}

If $T$ is additive with (continuously) differentiable components $ T_1 $ and $ T_2 $, then $ L_{(G_1,G_2)}( H_1, H_2 ) = L^{(1)}_{G_1}(H_1) + L^{(2)}_{G_2}(H_2) $ where $ L^{(i)}_{G_i} $ is the G\^ateaux derivative of $ T_i $, $ i = 1, 2 $. In general,
the linearization of the two sample statistic $ T( \wh{F}_{n_1}^{(1)}, \wh{F}_{n_2}^{(2)} ) $ induced by a (continuously) G\^ateaux differentiable functional $ T(F_1,F_2) $ is given by
\[
  L_{(F_1,F_2)}( \wh{F}_{n_1}^{(1)} - F_1, \wh{F}_{n_2}^{(2)} - F_2 ) = \frac{1}{n_1n_2} \sum_{i=1}^{n_1} \sum_{j=1}^{n_2} L_{(F_1,F_2)}( \delta_{X_{1i}} - F_1, \delta_{X_{2j}} - F_2 ),
\]
whereas for an addditive statistical functional this formula simplifies to  
\begin{equation}
	\label{LinearTermGateaux}
L_n^{(T)}  =   \frac{1}{n_1} \sum_{j=1}^{n_1} \psi_1( X_{1j} ) + \frac{1}{n_2} \sum_{j=1}^{n_2} \psi_2( X_{2j} ),
\end{equation}
where $ \psi_i(x) = L_{F_i}^{(i)}( \delta_{x} - F_i ) $, $ i = 1,2 $, such that $ E \psi_i(X_{i1} ) = 0 $, since, e.g.,
$ E( L_{F_1}^{(1)}( \delta_{X_{1i}} -  F_1 ) = L_{F_1}^{(1)}( E( \eins( X_{1i} \le \cdot))- F_1( \cdot ) ) = 0 $, $ i = 1, \dots, n_1 $, by linearity.

In general, the requirement that $T$ is G\^ateaux-differentiable is too weak to entail a central limit theorem. Therefore, from now on we assume, as in \cite{Shao1993}, where the one-sample setting is studied, that $ T_n = T( \wh{F}_{n_1}^{(1)}, \wh{F}_{n_2}^{(2)} )  $ is asymptotically normal,
\[
  \sqrt{n}( T_n - T(F_1,F_2) ) \stackrel{d}{\to} N( 0, \sigma^2(T) ),
\]
with asymptotic variance $ \sigma^2(T) $, where
\begin{equation}
\label{AsVarFormulaGateaux}
  \sigma^2(T) = \lim_{n \to \infty} \frac{n}{n_1} \Var( \psi_1( X_{11} ) ) + \frac{n}{n_1}  \Var( \psi_2( X_{21} ) )  = \lambda_1^{-1} \Var( \psi_1( X_{11} ) ) + \lambda_2^{-1} \Var( \psi_2( X_{21} ) ),
\end{equation}
if $ n \to \infty $ with $ n_i / n \to \lambda_i $, $ i = 1, 2 $. Consequently, (\ref{AsVarFormulaGateaux}) and (\ref{FormulaAsVarLinearization}) coincide.

The following lemma shows that under weak conditions the linearization $ L_n^{(T)} $ is asymptotically as required in the previous sections and clarifies the relationship between the functions $ \psi_1, \psi_2 $ in (\ref{LinearTermGateaux}) and the functions $ h_1, h_2 $ in (\ref{LinearTerm}). 

\begin{lemma} 
	\label{RelationshipLinearStats}
	Let $ T_n $ be a two-sample statistic induced by an additive functional.
	If $ E( \psi_i^2( X_{i1} ) ) < \infty $ and $ n_i/n \to \lambda_i $, $ i = 1, 2 $, then 
	\[
	L_n^{(T)} = \frac{1}{n} \left\{  \sum_{j=1}^{n_1} \lambda_1^{-1} \psi_1( X_{1j} ) + \sum_{j=1}^{n_2} \lambda_2^{-1} \psi_2( X_{2j} ) \right\} + R_n^{(T)},
	\]
	where $ E( \sqrt{n} R_n^{(T)} )^2 = o(1) $, as $ n \to \infty $, such that $ h_i = \lambda_i \psi_i $ , $ i = 1, 2 $.
	Therefore, one obtains the representation (\ref{LinearTerm}) when putting $ h_i(x) = \lambda_i^{-1} \psi_i(x) $, $ x \in \R $, $ i = 1, 2 $.
\end{lemma}

\begin{proof} We have $ L_n^{(T)} = \sum_{i=1,2}  \frac{1}{n}   \frac{n}{n_i} \sum_{j=1}^{n_i} \psi_i( X_{ij} ) $, and therefore
	\[
	  L_n^{(T)} = \frac{1}{n} \left[ \sum_{j=1}^{n_1} \lambda_1^{-1} \psi_1( X_{1j} ) + \sum_{j=1}^{n_2} \lambda_2^{-1}  \psi_2( X_{2j} )  \right] + R_n^{(T)}
	\]
	where $ \sqrt{n} R_n^{(T)} = \sum_{i=1,2} \sqrt{\frac{n_i}{n}}  \left( \frac{n}{n_i} - \lambda_i^{-1} \right) \frac{1}{\sqrt{n_i}}  \sum_{j=1}^{n_i} \psi_i( X_{ij} ) $ satisfies $ E( \sqrt{n} R_n^{(T)} )^2 = o(1) $, since $ E( \frac{1}{\sqrt{n_i}}  \sum_{j=1}^{n_i} \psi_i( X_{ij} ))^2 = O(1) $, as $ n_i \to \infty $, $ i = 1, 2 $.
\end{proof}

\begin{theorem}
\label{ConsistencyGateaux}
	Assume that $T $ is a continuously G\^ateaux differentiable two sample functional with 
	\[
	E( L_{(F_1,F_2)}( \delta_{X_1} - F_1, 0 ) ) = E (L_{(F_1,F_2)}( 0, \delta_{X_2} - F_2 ) ) = 0 
	\]
	and
	\[
	E( L_{(F_1,F_2)}( \delta_{X_1} - F_1, 0 )^2  ), E (L_{(F_1,F_2)}( 0, \delta_{X_2} - F_2 )^2 ) < \infty.
	\]
	Then
	\[
	  \wh{\tau}_i^2 \stackrel{a.s.}{\to} \Var( L_{(F_1,F_2)}( \delta_{X_{i1}} - F_i, 0 ),
	\]
	as $ n_i \to \infty $, $ i = 1, 2$, and therefore
	\[
	  \wh{\tau}_2^2 +  \wh{\tau}_2^2 \stackrel{a.s.}{\to} \sigma^2(T),
	\]
	as $ n \to \infty $, provided $ n_i/n \to \lambda_i $, $ i = 1, 2 $.
\end{theorem}

\begin{proof} The general method of the proof is as in \cite{Shao1993}.
	We show the assertion for $ \wh{\tau}_1^2 $. Observe that for $ i = 1, \dots, n_1 $ 
	\[
	\wh{\xi}_i = n T( \wh{F}_{n_1}^{(1)}, \wh{F}_{n_2}^{(2)} ) - (n-1) T( \wh{F}_{n_1,-i}^{(1)}, \wh{F}_{n_2}^{(2)}  ).
	\]
	Hence, since the sample variance is location-invariant,
	\begin{align*}
	\wh{\tau}_1^2 & = \frac{1}{n_1-1} \sum_{i=1}^{n_1} ( \wh{\xi}_{i} - \overline{\wh{\xi}}_{1:n_1} )^2 \\
	&= \frac{(n-1)^2}{n_1-1} \sum_{i=1}^{n_1} 
	\left(
	T( \wh{F}_{n_1,-i}^{(1)}, \wh{F}_{n_2}^{(2)} ) - \frac{1}{n_1} \sum_{k=1}^{n_1} T( \wh{F}_{n_1,-k}^{(1)}, \wh{F}_{n_2}^{(2)} ) 
	\right)^2
	\end{align*}
	Since  $ (n_1-1) \wh{F}_{n_1,-i}^{(1)} = n_1 \wh{F}_{n_1}^{(1)} - \delta_{X_{1i}} $, such that
   \[ \wh{F}_{n_1,-i}^{(1)} = \frac{n_1}{n_1-1} \wh{F}_{n_1}^{(1)}  - \frac{1}{n_1-1} \delta_{X_{1i}} = \wh{F}_{n_1}^{(1)} - \frac{1}{n_1-1} ( \delta_{X_{1i}} - \wh{F}_{n_1}^{(1)} ), \] 
   we obtain
	\[
	( \wh{F}_{n_1,-i}^{(1)}, \wh{F}_{n_2}^{(2)} ) = ( \wh{F}_{n_1}^{(1)}, \wh{F}_{n_2}^{(2)} ) - t_n
	\left( \delta_{ X_{1i} } - \wh{F}_{n_1}^{(1)}, 0  \right) 
	\]
	with $ t_n = 1/(n_1-1) $.  Because $ \max_{i=1,2} \| \wh{F}_{n_i}^{(i)} - F_i \|_\infty \to 0 $, $ n \to \infty $, a.s., and
	since $ T $ is continuously G\^ateaux differentiable, we may conclude that
	\begin{equation}
	\label{GCons} 
	\max_{1 \le i \le n_1} 
	\left| 
	\frac{ T( \wh{F}_{n_1,-i}^{(1)}, \wh{F}_{n_2}^{(2)} ) - T( \wh{F}_{n_1}, \wh{F}_{n_2} ) }{ t_n }
	- L_{(F_1,F_2)}( \delta_{X_{1i}} - \wh{F}_{n_1}^{(1)}, 0 ) \right|
	\to 0,
	\end{equation}
	as $ n \to \infty $, w.p. $1$. By linearity,
	\begin{align*}
	L_{(F_1,F_2)}( \delta_{X_{1i}} - \wh{F}_{n_1}^{(1)}, 0 ) 
	& = L_{(F_1,F_2)}( \delta_{X_{1i}}, 0 ) - \frac{1}{n_1} \sum_{j=1}^{n_1} L_{(F_1,F_2)}( \delta_{X_{1j}}, 0 ) 
	= Z_i - \overline{Z}_{n_1},
	\end{align*}
	if we put $ Z_i = L_{(F_1,F_2)}( \delta_{X_{1i}}, 0 ) $, $ i = 1, \dots, n_1 $, and $ \overline{Z}_{n_1} = \frac{1}{n_1} \sum_{j=1}^{n_1} Z_i $. Since the $ Z_i $, $ 1 \le i \le n_1 $, are i.i.d. centered with finite variance, the strong law of large numbers implies $ \frac{1}{n_1} \sum_{i=1}^{n_1} Z_i^r \to E Z_1^r $, a.s., $ r = 1, 2 $, such that 
	\begin{equation}
	\label{GCons3}
	  \frac{1}{n_1} \sum_{i=1}^{n_1} (Z_i - \overline{Z}_{n_1})^2 - \Var(Z_1) \to 0,
	\end{equation}
	as $ n \to \infty $, a.s.. (\ref{GCons}) means that
	\begin{equation}
	\label{GCons2}
	\max_{1 \le i \le n_1} 
	\left| 
	(n_1-1) [  T( \wh{F}_{n_1}, \wh{F}_{n_2} ) - T( \wh{F}_{n_1,-i}^{(1)}, \wh{F}_{n_2}^{(2)} ) ]
	-    (Z_i - \overline{Z}_{n_1})  \right|
	\to 0,
	\end{equation}
	as $ n \to \infty $, a.s.. Combining this with $ | \overline{Z}_n - E(Z_1) | \stackrel{a.s.}{\to} 0 $, as $ n \to \infty $, we arrive at
	\[ 
	  \max_{1 \le i \le n_1} | 	(n_1-1) [  T( \wh{F}_{n_1}, \wh{F}_{n_2} ) - T( \wh{F}_{n_1,-i}^{(1)}, \wh{F}_{n_2}^{(2)} ) ] - (Z_i - \mu) | \stackrel{a.s.}{\to} 0,
	\]
	as $ n \to \infty $. It follows that the sample variance of the random variables
	$ V_{n_1,-i} =  (n_1-1) T( \wh{F}_{n_1,-i}^{(1)}, \wh{F}_{n_2}^{(2)} ) $, $ 1 \le i \le n_1 $, converges a.s. to $ \Var( Z_1 ) $ as well.
	But this implies
	\[
	  \wh{\tau}_1^2 = \left( \frac{n-1}{n_1-1} \right)^2 \frac{1}{n_1-1} \sum_{i=1}^{n_1} \left( V_{n_1,-i} - \frac{1}{n_1} \sum_{k=1}^{n_1} V_{n_1,-k}   \right)^2
	  \stackrel{a.s.}{\to} \lambda_1^{-2} \Var( Z_1 ),
	\]
	as $ n \to \infty $. Recalling the representation (\ref{AsVarFormulaGateaux}) the a.s. convergence
	of $ \wh{\tau}_1^2 + \wh{\tau}_2^2 \to \sigma^2(T) $, as $n \to \infty $, follows.
\end{proof}

\begin{corollary}	
\label{CorollaryGateaux}
	Under the conditions of Theorem~\ref{ConsistencyGateaux} the jackknife variance estimator defined in (\ref{JackknifeEst2Sample}) is consistent for $ \sigma^2(T) = \lambda_1^2 \Var( h_1( X_{11} )) + \lambda_2^2 \Var( h_2( X_{21} )) $, since 
	\[
	\wh{\tau}_i^2 \stackrel{a.s.}{\to} E( h_i^2(X_{i1} ) ),
	\]
	as $ n \to \infty $, for $ i = 1, 2$.
\end{corollary}

\begin{proof}
	For an additive functional, $ L_{(F_1,F_2)}(H_1, H_2) = L_{F_1}^{(1)}( H_1 ) + L_{F_2}^{(2)}( H_2 ) $, such that
	$ Z_1 = L_{F_1}( \delta_{X_{11}} - F_1 ) = \psi_1( X_{11} ) $ and hence $ \Var( Z_1 ) = \Var( \psi_1( X_{11} ) ) $.
	As shown in Lemma~\ref{RelationshipLinearStats}, $ \psi_1 = \lambda_1 h_1 $, where $ h_1 $ is as in (\ref{LinearTerm}), leading to $ \wh{\tau}_1^2 \stackrel{a.s.}{\to} \lambda_1^{-2} \Var(Z_1) = \Var( h_1(X_{11} ) ) $.
\end{proof}

To summarize, under G\^ateaux differentiability the jackknife variance estimator works, provided asymptotic normality holds. Let us now study the case of a two-sample statistic induced by a Fr\'echet-differentiable statistical functional.

Let $ \rho_1 $ be a metric defined on $ \calD $ and introduce the metric 
\[
  \rho( (G_1, G_2), (H_1, H_2) ) = \rho_1( G_1, H_1 ) + \rho_2( G_2, H_2 ),
\]  
for $ (G_1,H_1), (G_2, H_2) \in \calD \times \calD $, on $ \calD \times \calD $. 

\begin{definition}
	A functional $ T: \calF \times F \to \R $ is called {\em continuously Fr\'echet-differentiable} at $ (G_1,G_2) \in \calF \times \calF $ with respect to $ \rho_1 $, if $ T $ is Fr\'echet-differentiable at $ (F_1, F_2) $, i.e. there is a linear functional $ L_{(F_1,F_2)} $ on $ \calD $, such that for all $ (C_1,C_2) \in \calC = \{ A \subset \calD : A \ \text{bounded} \} $
	\[
	\lim_{t \to 0} \sup_{D \in \calC} \frac{ T( (G_1, G_2) + t D ) - T(G_1,G_2) - L_{(G_1,G_2)}( t D ) }{t} = 0,
	\]
	and $ \rho(G_k,G) \to 0 $ and $ \rho(H_k,H) \to 0 $, as $ k \to \infty $, entail
	\[
	\lim_{k \to \infty} \frac{ T( H_k ) - T(G_k) - L_G( H_k - G_k )}{ \rho(H_k, G_k) } = 0.
	\]
\end{definition}

Proper examples for the choice of $ \rho_1 $, e.g. to ensure that sample quantiles are continuously Fr\'echet differentiable w.r.t. $ \rho_1 $, are discussed in \cite{Shao1993}.

\begin{theorem}
	Suppose that $ T $ is a continuously Fr\'echet-differentiable statistical functional w.r.t a metric $ \rho_1 $. Assume that
	\begin{equation}
	\label{JackGeneralAss-1}
		\rho_1( \wh{F}_{n_i}^{(i)}, F_i ) \stackrel{a.s.}{\to} 0,
	\end{equation}
	as $ n_i \to \infty $, $i = 1, 2 $, and
	\begin{equation}
	\label{JackGeneralAss-2}
		\sum_{j=1}^{n_i} \rho_1^2( \wh{F}_{n_i,-j}^{(i)}, \wh{F}_{n_i}^{(i)} ) = O( n_i^{-1} ),
	\end{equation}
	a.s., for $i = 1, 2$. Then the assertions of Theorem~\ref{ConsistencyGateaux} and Corollary~\ref{CorollaryGateaux} hold true.
\end{theorem}

\begin{proof}
	Put 
	\[
	  R_{ni} = T( \wh{F}_{n_1,-i}^{(1)}, \wh{F}_{n_2}^{(2)} ) - T( \wh{F}_{n_1}^{(1)},  \wh{F}_{n_1}^{(2)} ) 
	  - L_{(F_1,F_2)}( \wh{F}_{n_1,-i}^{(1)} - \wh{F}_{n_1}^{(1)}, 0  )
	\]
	for $ i = 1, \dots, n_1 $, and $ \overline{R} = \frac{1}{n_1} \sum_{j=1}^{n_1} R_{nj} $. Here
	\[
		L_{(F_1,F_2)}( \wh{F}_{n_1,-i}^{(1)} - \wh{F}_{n_1}^{(1)}, 0  )
		= \frac{1}{n_1-1} \sum_{j=1, j \not=i}^{n_1} Z_j - \overline{Z},
	\]
	with $ Z_i = L_{(F_1,F_2)}( \delta_{X_{1i}} - F_1, 0 ) $, $ i = 1, \dots, n_1 $, and $ \overline{Z} = \frac{1}{n_1} \sum_{i=1}^{n_1} Z_i $. Recall that
	\[
	  \wh{\tau}_1^2 = \frac{(n-1)^2}{n_1-1} \sum_{i=1}^{n_1} \left( T( \wh{F}_{n_1,-i}^{(1)}, \wh{F}_{n_2}^{(2)} ) - \overline{T( \wh{F}_{n_1,-\cdot}^{(1)}, \wh{F}_{n_2}^{(2)} )} \right)^2,
	\]
	where $  \overline{T( \wh{F}_{n_1,-\cdot}^{(1)}, \wh{F}_{n_2}^{(2)} )} = \frac{1}{n_1} \sum_{i=1}^{n_1} T( \wh{F}_{n_1,-i}^{(1)}, \wh{F}_{n_2}^{(2)} ).$ 
	Using
	\[
	  T( \wh{F}_{n_1,-i}^{(1)}, \wh{F}_{n_2}^{(2)} ) - \overline{T( \wh{F}_{n_1,-\cdot}^{(1)}, \wh{F}_{n_2}^{(2)} )}
	  = R_{ni} - \overline{R} +  \frac{1}{n_1-1} \sum_{j=1, j \not=i}^{n_1} Z_j - \overline{Z}
	\]
	we obtain
	\begin{align*}
	   \wh{\tau}_1^2 & =  \frac{(n-1)^2}{(n_1-1)^3} \sum_{i=1}^{n_1} (Z_i - \overline{Z})^2
	    +  \frac{(n-1)^2}{n_1-1} \sum_{i=1}^{n_1} ( R_{ni} - \overline{R} )^2 \\
	    & \qquad + 2  \frac{(n-1)^2}{n_1-1} \sum_{i=1}^{n_1} R_{ni} \frac{1}{n_1-1} \sum_{j=1,j \not=i }^{n_1} Z_j - \overline{Z}. 
	\end{align*}
	Since $ \frac{(n-1)^2}{(n_1-1)^3} = O( \frac{1}{n-1} ) $ and $ \frac{(n-1)^2}{n_1-1} = O( n-1 ) $, the same arguments as in \cite{Shao1993} entail that it suffices to show that
	\begin{equation}
	\label{Th2ToShow}
	  (n_1-1) \sum_{i=1}^{n_1} R_{ni}^2 \stackrel{a.s.}{\to} 0,
	\end{equation}
	as $ n_1 \to \infty $. First, observe that $ \rho( (\wh{F}_{n_1,-i}^{(1)}, 0 ), (F_1,0 ) ) = \rho_1( \wh{F}_{n_1,-i}^{(1)}, F_1 ) $
	and 
	\[
	  \max_{1 \le i \le n_1} \rho_1( \wh{F}_{n_1,-i}^{(1)}, F_1 ) \le \rho_1( \wh{F}_{n_1}^{(1)}, F_1 ) + \sqrt{ \max_{1 \le i \le n_1} \rho_1^2( \wh{F}_{n_1,-i}^{(1)}, \wh{F}_{n_1}^{(1)} ) } \stackrel{a.s}{\to} 0,
	\]
	as $ n_1 \to \infty $ by (\ref{JackGeneralAss-1}). Since $ T $ is continuously Fr\'echet-differentiable, 
    \[
      \frac{ T( \wh{F}_{n_1,-i}^{(1)}, \wh{F}_{n_2}^{(2)} ) - T( \wh{F}_{n_1}^{(1)},  \wh{F}_{n_2}^{(2)} )  
      	- L_{(F_1,F_2)}( \wh{F}_{n_1,-i}^{(1)} - \wh{F}_{n_1}^{(1)}, 0  )  }{ \rho_1( \wh{F}_{n_1,-i}^{(1)}, \wh{F}_{n_1}^{(1)} ) } = o(1),
    \]
    as $ n_1 \to \infty $, a.s.. Therefore, for any $ \varepsilon > 0 $ there exists $ n_0 \in \N $ such that for $ n \ge n_0 $
	\[
	  | T( \wh{F}_{n_1,-i}^{(1)}, \wh{F}_{n_2}^{(2)} ) - T( \wh{F}_{n_1}^{(1)},  \wh{F}_{n_1}^{(2)} )  
	  - L_{(F_1,F_2)}( \wh{F}_{n_1,-i}^{(1)} - \wh{F}_{n_1}^{(1)}, 0  ) | \le \varepsilon \rho_1( \wh{F}_{n_1,-i}^{(1)}, \wh{F}_{n_1}^{(1)} )
	\]
	It follows that $ R_{ni}^2 \le \varepsilon^2 \rho_1^2( \wh{F}_{n_1,-i}^{(1)}, \wh{F}_{n_1}^{(1)} ) $, such that
	\[
	  (n_1-1) \sum_{i=1}^{n_1} R_{ni}^2 \le \varepsilon^2 (n_1-1) \sum_{i=1}^{n_1} \rho_1^2( \wh{F}_{n_1,-i}^{(1)}, \wh{F}_{n_1}^{(1)} ) = O( \varepsilon ),
	\]
	as $ n_1 \to \infty $, a.s., by virtue of (\ref{JackGeneralAss-2}), which completes the proof.
\end{proof}

\section{Review of Estimation under ordered variances}
\label{sec: Estimation}

Let $X_{ij}, i=1,2, j= 1, \dots, n_i(\geq 2)$ be independent observations from the normal distribution with mean $\mu_i$ and variance $\sigma_i^2$, where both $\mu_i$ and $\sigma_i^2$ are unknown. 
Also let
\[
\overline{X}_i = \frac{1}{n_i} \sum_{j=1}^{n_i} X_{ij}, \qquad S_i^2 = \frac{1}{n_i} \sum_{j=1}^{n_i} (X_{ij}- \overline{X}_i)^2, \qquad \wt{S}_i^2 = \frac{n_i}{n_i-1} S_i^2,
\]
be the sample means, sample variances and the associated unbiased variance estimators, which are frequently used for estimation of $ \mu_i $ and $ \sigma_i^2 $, $ i = 1, 2 $. When needed, we shall indicate the dependence on the sample sizes and write $ S_{n_i,i}^2 $ for $ S_i^2 $, $ i = 1, 2$.

When $\sigma_i^2 , i=1,2$ are known, to estimate the common mean, $ \mu = \mu_1 = \mu_2 $,   one may use the unbiased minimum variance estimator
\[
\wh{\mu}^{(0)} = \frac{ n_1 \sigma_2^2 }{ n_1 \sigma_2^2 + n_2 \sigma_1^2 } 
\overline{X}_1 + \frac{ n_2 \sigma_1^2 }{ n_1 \sigma_2^2 + n_2 \sigma_1^2 } 
\overline{X}_2.
\]
When variances are unknown,  Graybill-Deal (GD)  proposed for the case of no order restriction the estimator
\[
\wh{\mu}^{(1)} = \frac{ n_1 \wt{S}_2^2 }{ n_1 \wt{S}_2^2 + n_2 \wt{S}_1^2 } 
\overline{X}_1 + \frac{ n_2 \wt{S}_1^2 }{ n_1 \wt{S}_2^2 + n_2 \wt{S}_1^2 } 
\overline{X}_2.
\]
Here the sample averages are weighted with {\em random} weights forming a convex combination as 
	\[
	\wh{\mu} (\gamma) = \gamma \overline{X}_1+ ( 1-\gamma )  \overline{X}_2.
	\]
	where $\gamma$ is a function of $S_1^2, S_2^2$ and possible $(\overline{X}_1-\overline{X}_2)^2$.

\cite{Kubokawa1989} has introduce a broad class of common mean estimators 
with $\gamma$ given by 
\[
\gamma_{\psi} = 1 - \frac{ a }{ b R \psi( \wt{S}_1^2, \wt{S}_2^2, (\overline{X}_1 - \overline{X}_2)^2 ) }
\]
where $ R = ( b \wt{S}_2^2 + c ( \overline{X}_1 - \overline{X}_2 )^2 ) / \wt{S}_1^2 $,
	$ \psi $ a positive function and $ a, b, c \ge 0 $ constants. For suitably chosen $\psi, a, b$ and $c$. Kobokawa has given a sufficient condition on $n_1$ and $n_2$ so that $\wh{\mu}(\gamma_{\psi})$ is closer to $\mu$ than $\overline{X}_1$ , see \cite{Pitman1937}. 
Such estimators have also been studied by several authors assuming Gaussian samples, see e.g. \cite{BrownCohen1974} 
and \cite{Bhattacharya1980}, amongst others.

When there is an order restriction between both variances,
\cite{MehtaGurland1969} proposed three convex combinations estimators for small samples and compare the efficiencies of proposed estimators with GD estimator. 
When order constraints on the variances apply, the question arises whether one can improve upon the above proposals. There is a rich literature on the general problem of estimation for constrained parameter spaces and we refer the reader to the monograph of \cite{vanEeden2006}.

For the common mean estimation problem studied here, the following results have been obtained, again under the assumption of normality: 
When there is a constraint of ordered variances, $ \sigma_1^2 \le \sigma_2^2 $,
\cite{Elfessi1992}  proposed
\[
\wh{\mu}^{(2)} 
= \wh{\mu}^{(1)} \eins( \wt{S}_1^2 \le \wt{S}_2^2 ) +
\left[ \frac{n_1}{n_1+n_2} \overline{X}_1 + \frac{n_2}{n_1+n_2} 
\overline{X}_2 \right] \eins( \wt{S}_1^2 > \wt{S}_2^2 ),
\]
and showed that $ \wh{\mu}^{(2)} $ also stochastically dominates the GD estimator.  
For the special case of balanced sample sizes, $ n_1 = n_2 $, \cite{Elfessi1992} 
also proposed the estimator
\[
\wh{\mu}^{(3)} 
= \wh{\mu}^{(1)} \eins( \wt{S}_1^2 \le \wt{S}_2^2 ) +
\left[ \frac{\wt{S}_1^2}{\wt{S}_1^2 + \wt{S}_2^2} \overline{X}_1 + \frac{\wt{S}_2^2}{ \wt{S}_1^2 + \wt{S}_2^2} 
\overline{X}_2 \right] \eins( \wt{S}_1^2 > \wt{S}_2^2 ),
\]
which stochastically dominates the GD estimator as well. 


Chang et al. (2012) 
have shown, see their Theorem~2.1, that for ordered variances, $ \sigma_1^2 \le \sigma_2^2 $, any such estimator can be further improved by replacing $ \gamma_n $ by some appropriately chosen $ \wt{\gamma} $ if $ \gamma_n < n_1/n $. This means, one uses
\[
\gamma^+_n = \left\{
\begin{array}{ll}
\gamma_n, & \gamma_n \ge n_1/n, \\
\wt{\gamma}, & \gamma_n < n_1/n, 
\end{array}
\right.
\]
instead of $ \gamma_n $. 
Chang et al. (2012)
proved that the estimator $ \wh{\mu}_n(\gamma^+_n) $ stochastically dominates $ \wh{\mu}_n(\gamma_n) $, if $ \gamma_n < n_1/n $ with positive probability and $ \wt{\gamma} $ satisfies the 
constraints
\begin{equation}
	\label{ConditionChang}
	n_1/n \le \wt{\gamma} \le 2 n_1/n - \gamma_n.
\end{equation}
Similar results are obtained by Chang and Shinozaki (2015) under Pitman closeness criterion.

All results discussed above have been obtained assuming Gaussian samples. Motivated by the above findings, we consider, within a purely nonparametric setup, common mean estimators of the form
\begin{equation}
\label{DefEst}
 \wh{\mu}_n(\gamma) = \gamma_n \overline{X}_1 + (1-\gamma_n) \overline{X}_2
\end{equation}
where the weight
\[
  \gamma_n = \gamma( n_1/n, n_2/n, \wt{S}_1, \wt{S}_2, \overline{X}_1,  \overline{X}_2 ) 
\]
is a function of the the sample fractions $ n_1/n, n_2/n $ and the statistic $ (\wt{S}_1, \wt{S}_2, \overline{X}_1,  \overline{X}_2 ) $, which is the sufficient statistic under normality. To establish the required theoretical results, we need the following weak assumption on the weights $ \gamma_n $, which is satisfied by the above choices.

\textbf{Assumption ($\Gamma$):} $ X_{ij} \sim F_i $, $ 1 \le j \le n_i $,  $ i = 1,2 $, are independent random samples with common means $ \mu = \mu_1 = \mu_2 $ and arbitrary variances $ \sigma_1^2 $ and $ \sigma_2^2 $. The random weights $ \gamma_n = \gamma_n( n_1/n, n_2/n, \wt{S}_1^2, \wt{S}_2^2, \overline{X}_1, \overline{X}_2 ) $ are either of the form
\[
  \gamma_n = \left\{
    \begin{array}{ll}
      \gamma^\le, \qquad & \wt{S}_1^2 \le \wt{S}_2^2, \\
      \gamma^>, \qquad & \wt{S}_1^2 > \wt{S}_2^2,
    \end{array}
  \right.
\]
for two functions $ \gamma^\le = \gamma^\le ( n_1/n, n_2/n, \wt{S}_1^2, \wt{S}_2^2, \overline{X}_1, \overline{X}_2) $ and $ \gamma^> = \gamma^>( n_1/n, n_2/n, \wt{S}_1^2, \wt{S}_2^2, \overline{X}_1, \overline{X}_2) $, where $ \gamma^\le( \cdot ) $ and $ \gamma^>(\cdot) $ are three times continuously differentiable functions with bounded third partial derivatives, or of the form 
\[
\gamma_n = \left\{
\begin{array}{ll}
\gamma^\le, \qquad & S_1^2 \le S_2^2, \\
\gamma^>, \qquad & S_1^2 > S_2^2,
\end{array}
\right.
\]
for two functions $ \gamma^\le = \gamma^\le ( n_1/n, n_2/n, S_1^2, S_2^2, \overline{X}_1, \overline{X}_2) $ and $ \gamma^> = \gamma^>( n_1/n, n_2/n, S_1^2, S_2^2, \overline{X}_1, \overline{X}_2) $, where $ \gamma^\le( \cdot ) $ and $ \gamma^>(\cdot) $ are three times continuously differentiable functions with bounded third partial derivatives.

\begin{example}
	Observe that
	\[
	  \wh{\mu}^{(1)} = \frac{ \wt{S}_2^2 }{ \wt{S}_2^2 + \frac{n_2}{n} \frac{n}{n_1} \wt{S}_1^2 } \overline{X}_1
	  + \frac{ \wt{S}_1^2 }{ \frac{n_1}{n} \frac{n}{n_2} \wt{S}_2^2 + \wt{S}_1^2 } \overline{X}_2 
	  = \gamma_4^{\le}( n_1/n, n_2/n,\wt{S}_1^2, \wt{S}_2^2, \overline{X}_1, \overline{X}_2 ) 
	\]
	and 
	\[ 
		\frac{n_1}{n_1+n_2} \overline{X}_1 + \frac{n_2}{n_1+n_2} \overline{X}_2
		= \frac{1}{ 1 + \frac{n_2}{n} \frac{n}{n_1} } \overline{X}_1 + \frac{1}{ \frac{n_1}{n} \frac{n}{n_2} + 1} \overline{X}_2
		= \gamma_4^>( n_1/n, n_2/n,\wt{S}_1^2, \wt{S}_2^2, \overline{X}_1, \overline{X}_2 ),
	\]
	where the functions
	\begin{align*}
	  \gamma_4^{\le}( a, b, s, t, \mu, \nu ) & =  \mu \frac{t}{ t + s (b/ a)} + \nu \frac{s}{t (a/b) + s},  \\
	  \gamma_4^{>}( a, b, s, t, \mu, \nu )  & = \mu \frac{1}{1+b/a} + \nu \frac{1}{a/b + 1},
	\end{align*}
	are defined for $ (a,b,s,t,\mu, \nu) \in (0,1)^2 \times (0,\infty)^2 \times \R^2 $. It is easily seen that all partial derivatives of order three exist and are bounded on compact subsets of the domain. For example,
	\begin{align*}
		\frac{\partial \gamma^{\le}}{\partial t} & = -a \nu s / (b (a t / b + s) ^ 2) + \mu  (-t / (b s / a + t) ^ 2 + 1 / (b s / a + t)), \\
		\frac{\partial^2 \gamma^{\le}}{\partial t^2} & =  2 a ^ 2 \nu s / (b ^ 2 (a t / b + s) ^ 3) + \mu (2 t / (b s / a + t) ^ 3 - 2 / (b s / a + t) ^ 2), \\
		\frac{\partial^2 \gamma^{\le}}{\partial t^2 \partial s} & =2 a ^ 2 \nu (-3 s / (a t / b + s) ^ 4 + 1 / (a t / b + s) ^ 3) / b ^ 2 + \mu (4 b / (a (b s / a + t) ^ 3) - 6 b t / (a (b s / a + t) ^ 4)),
	 \end{align*}
	 where all expressions appearing in a denominator are positive. Therefore, $ \wh{\gamma}^{(4)} $ attains the representation
	\[
	  \wh{\gamma}^{(4)} = \gamma_4^{\le}( n_1/n, n_2/n,\wt{S}_1^2, \wt{S}_2^2, \overline{X}_1, \overline{X}_2 )  \eins( \wt{S}_1^2 \le \wt{S}_2^2 ) + \gamma_4^>( n_1/n, n_2/n,\wt{S}_1^2, \wt{S}_2^2, \overline{X}_1, \overline{X}_2 ) \eins( \wt{S}_1^2 > \wt{S}_2^2 )
	\]
	and satisfies Assumption $(\Gamma)$.
\end{example}

\begin{remark}
As discussed above in greater detail, we have in mind the case of ordered variances and our results aim at contributing to the problem of common mean estimation under ordered variances. But since it turns out that many results also hold for unequal variances without order restriction, we omit the order restriction in Assumption $ (\Gamma) $.
\end{remark}

\begin{remark}
  Additional assumptions on the underlying distributions will be stated where needed.
\end{remark}

\begin{remark}
\label{RemarkOneSample}
  When discussing the case of equal sample sizes $ N = n_1 = n_2 $, we shall index all quantities by $N$ instead of $n$. Further, we may and will redefine
  $ \gamma_N $ as well as $ \gamma^\le $ and $ \gamma^> $ to be functions
  of
  \[
     \wh{\theta}_N = ( N/n, \wt{S}_1^2, \wt{S}_2^2, \overline{X}_1, \overline{X}_2 )',
  \]
  which converges a.s. under Assumption ($ \Gamma $) to
  \[
    \theta = (1/2, \sigma_1^2, \sigma_2^2, \mu_1, \mu_2 )',
  \]
  as $ N \to \infty $.
\end{remark}

\section{Variance estimation and asymptotic distribution theory for common mean estimators}

\label{sec: jackknife common mean}

As already indicated in the introduction, there is a lack of results about the estimation of the variance of the common mean estimators discussed in the literature. For the case of normal populations, \cite{Nair1980} calculated the variance of the GD estimator for two populations and \cite{Voinov1984} extended the result to the case of several samples. \cite{MehtaGurland1969} gave formulas for the variances of their common mean estimators. The issue of unbiased estimation of  the variance for the GD estimator has been studied in \cite{Voinov1984} and \cite{Sinha1985}. All of those results, however, heavily rely on the normal assumption.

Therefore, we propose to use the nonparametric jackknife variance estimator studied in the previous section, which is applicable for a wide class common mean estiomators defined by convex combinations of the sample means with random weights. We shall show that the jackknife is weakly consistent and asymptotically unbiased under fairly weak conditions without requiring normally distributed observations.

In order to make the dependence on the data explicit, denote the common mean estimator by
\[
  \wh{\mu}_n(\gamma) = \wh{\mu}_{n_1,n_2}( X_{11}, \dots, X_{1n_1}; X_{21}, \dots, X_{2n_2} ).
\]
The leave-one-out estimates corresponding to the first sample are given by
\[
  \wh{\mu}_{n,-i}^{(1)} = \wh{\mu}_{n_1-1,n_2}( X_{11}, \dots, X_{1,i-1}, X_{1,i+1}, \dots, X_{1n_1};
   X_{21}, \dots, X_{2n_2} ),
\]
$ i = 1, \dots, n_1 $, and those for the second sample are
\[
  \wh{\mu}_{n,-i}^{(2)} = \wh{\mu}_{n_1,n_2-1}( X_{11},  \dots, X_{1n_1};
   X_{21}, \dots,   X_{2,i-1}, X_{2,i+1}, \dots, X_{2n_2} ),
\]
$ i = 1, \dots, n_2 $. Now the jackknife variance estimator of $ \Var( \wh{\mu}_n(\gamma ) ) $ is given by
\[
  \wh{\Var}( \wh{\mu}_n(\gamma) ) = \frac{n_1}{n} \wh{\tau}_n^{(1)} + \frac{n_2}{n} \wh{\tau}_n^{(2)},
\]
with
\begin{align*}
  \wh{\tau}_n^{(i)}  & = \frac{1}{n_i} \sum_{j=1}^{n_i} \left( \wh{\mu}_{n,-j}^{(i)} - \overline{\wh{\mu}}_{n,\bullet}^{(i)} \right)^2
\end{align*}
where $ \overline{\wh{\mu}}_{n,\bullet}^{(i)} = n_i^{-1} \sum_{j=1}^{n_i} \wh{\mu}_{n,-j}^{(i)} $, $ i = 1, 2 $.

The simplified jackknife variance estimator for the special case of equal sample sizes is as follows:
Again, to make the dependence on the data explicit, let us write
\[
  \wh{\mu}_N( \gamma ) = \wh{\mu}_N( Z_1, \dots, Z_N )
\]
where
\[
  Z_i = (X_{1i}, X_{2i})', \qquad i = 1, \dots, N.
\]
Consider the leave-one-out estimates
\[
  \wh{\mu}_{N,-i}( \gamma ) = \wh{\mu}_N( Z_1, \dots, Z_{i-1}, Z_{i+1}, \dots, Z_N ),
\]
for $ i = 1, \dots, N $. Then the jackknife estimator for the asymptotic variance of $ \wh{\mu}_N(\gamma) $,
\[
  \sigma^2(\gamma) = \sigma_N^2( \wh{\mu}_N(\gamma) ) = \Var( \sqrt{N} \wh{\mu}_N(\gamma) )
\]
is defined by
\begin{equation}
\label{Jack0}
\wh{\sigma}_N^2( \gamma ) = (N-1) \sum_{i=1}^N 
( \wh{\mu}_{N,-i} - \overline{\wh{\mu}}_{N, \bullet } )^2,
\end{equation}
and the jackknife variance estimator of $ \Var( \wh{\mu}_N(\gamma) ) $ is \begin{equation}
\label{Jack}
\wh{\Var}( \wh{\mu}_N(\gamma) ) = \frac{N-1}{N} \sum_{i=1}^N 
( \wh{\mu}_{N,-i} - \overline{\wh{\mu}}_{N, \bullet } )^2,
\end{equation}
where $ \overline{\wh{\mu}}_{N,\bullet}  = N^{-1} \sum_{i=1}^N \wh{\mu}_{N,-i} $.

The consistency of the jackknife is now established by invoking the key results of the previous section,
namely by proving that the common mean estimator is asymptotically linear with an appropriate remainder term. For simplicity of exposition, we state and prove the result for the case of equal sample sizes. The proof only uses elementary arguments, but it is  long and technical. It is therefore provided in Appendix~\ref{AppendixB}.

\begin{theorem} 
\label{Theorem_Taylor}
Suppose that $ X_{ij} \sim F_i $, $ j = 1, \dots, N $, are two i.i.d. samples
with distribution functions $ F_1, F_2 $ satisfying $ E( X_{11}^{12} ) < \infty $, $ E( X_{21}^{12} ) < \infty $. Then the following assertions hold.
\begin{itemize}
\item[(i)] We have
  \begin{align*}
    \wh{\mu}_N(\gamma) - \mu 
    & = \biggl[ (\nabla [ \gamma^\le(\theta) \mu] + \nabla [(1-\gamma^\le(\theta)) \mu ]) \eins_{\{ \sigma_1^2 < \sigma_2^2 \}} \\
    & \quad +
       (\nabla [\gamma^>(\theta) \mu] + \nabla [(1-\gamma^>(\theta)) \mu])  \eins_{\{ \sigma_1^2 > \sigma_2^2 \}} \biggr]( \wh{\theta}_N - \theta) + R_N^\gamma,
  \end{align*}
  for some remainder term $ R_N^\gamma $ with $ N E(R_N^\gamma)^2 = O(1/N) $
  and $ N^2 E( R_N - R_{N-1} )^2 = o(1) $, as $ N \to \infty $.
\item[(ii)] For ordered variances, i.e. if either $ \sigma_1^2 < \sigma_2^2 $ it holds
  \[
    \wh{\mu}_N(\gamma) - \mu = \gamma(\theta) ( \overline{X}_1 - \mu ) + (1-\gamma(\theta)) ( \overline{X}_2 - \mu ) + R_N,
  \]
  and for $ \sigma_1^2 > \sigma_2^2 $
  \[
  \wh{\mu}_N(\gamma) - \mu = (1-\gamma(\theta)) ( \overline{X}_1 - \mu ) + \gamma(\theta) ( \overline{X}_2 - \mu ) + R_N,
  \]  
  with $ N E( R_N^2 ) = O(1/N) $.
\item[(iii)] The jackknife variance estimator is consistent, 
\[ 
  \frac{\wh{\sigma}_N^2( \wh{\mu}_N )}{\sigma^2(\gamma)} = 1 + o_P(1),
\]
as $ N \to \infty $, and asymptotically unbiased,
\[
  \frac{E( \wh{\sigma}_N^2( \wh{\mu}_N ))}{\sigma^2(\gamma)} = 1 + o(1),  
\]
as $ N \to \infty $, and the same applies to $ \wh{\sigma}_N(\gamma) $.
\end{itemize} 
\end{theorem}

\begin{remark} For unequal sample sizes the assertions follow under the additional condition
\begin{equation}
\label{CondSampleSizes}
  \lim_{\min(n_1,n_2) \to \infty} \frac{n_i}{n} = \lambda_i, \qquad i = 1,2,
\end{equation}
for $ \lambda_1, \lambda_2 \in (0, 1) $.  Then the estimates  
 (\ref{HigherOrderEstimateSampleVariance}) and (\ref{HigherOrderEstimateSampleVariance2}) hold, since they only require the decomposition (\ref{SampleVarianceAsLinear}) and the estimates (\ref{PropRemainderSampleVar}) given in Lemma~\ref{SampleVarAsLinear}.
\end{remark}

Let us now briefly discuss the asymptotic normality of the class of estimators under investigation, which does not require ordered variances. Put
\[
  \gamma = \gamma^\le \eins_{\{ \sigma_1^2 < \sigma_2^2\}} + \gamma^> \eins_{\{ \sigma_1^2 > \sigma_2^2 \}}.
\]
Clearly, for deterministic $ \gamma_N $ the CLT holds. More generally, for random weights satisfying Assumption ($\Gamma$), i.e. if $ \gamma_N $ may depend on $ \wh{\theta}_N $, we have the following result. 

\begin{theorem} 
\label{CLT}
Under Assumption $(\Gamma)$ and (\ref{CondSampleSizes}) the common mean estimator $ \wh{\mu}_n(\gamma) $ satisfies the central limit theorem
\[
  \sqrt{n}[ \wh{\mu}_n( \gamma ) - \mu ] \stackrel{d}{\to} 
  N( 0, \sigma^2(\gamma) ),
\]
as $ n \to \infty $, where the asymptotic variance is given by
\begin{equation}
\label{AsVar}
 \sigma^2( \gamma ) = \gamma^2 \lambda_1^{-1} \sigma_1^2 + (1-\gamma)^2 \lambda_2^{-1} \sigma_2^2.
\end{equation}
\end{theorem}

\begin{proof} 
We consider the case $ \sigma_1^2 \le \sigma_2^2 $. Then $ \gamma(\cdot) = \gamma^\le(\cdot) $ is continuous, such that  
\[ b_n(\wh{\theta}_n) = ( \gamma(\wh{\theta}_n), 1-\wh{\gamma}(\wh{\theta}_n ) )' \to b(\theta) = (\gamma^\le(\theta), 1-\gamma^\le(\theta) )',
\]
as $ n \to \infty $ a.s.. Clearly,
\[ \sqrt{n} (\overline{X}_i - \mu ) \stackrel{d}{\to} N(0, \lambda_i^{-1} \sigma_i^2 ),\] 
as $ n \to \infty $, for $ i = 1, 2 $. Notice that $ n = n_1 + n_2 \to \infty $
and  $ n_i / n \to \lambda_i $, as $ n \to \infty $, imply
$ n_i = (n_i/n) n \to \infty $, as $n \to \infty $, for $ i = 1, 2 $. Let us show that
\begin{equation}
\label{CLT_Vn}
  V_n = \sqrt{n} \left( \begin{array}{cc} \overline{X}_1 - \mu \\ \overline{X}_2 - \mu \end{array} \right) \stackrel{d}{\to} Z = \left( \begin{array}{cc} \lambda_1^{-1/2} \sigma_1 Z_1 \\ \lambda_2^{-1/2} \sigma_2 Z_2 \end{array} \right),
\end{equation}
as $ n \to \infty $, where $ Z_1, Z_2 $ are independent standard normal random variables.
We apply the Cram\'er-Wold device. For $ \rho_1, \rho_2 \in \R $, we have
\begin{align*}
  T_n(\rho_1, \rho_2) & = \rho_1 \sqrt{n}( \overline{X}_1 - \mu ) + \rho_2 \sqrt{n}( \overline{X}_2 - \mu )  \\
  & = \rho_1 \sqrt{\frac{ n }{ n_1 } } \frac{1}{\sqrt{n_1}} \sum_{j=1}^{n_1} (X_{1j} - \mu )
  + \rho_2 \sqrt{\frac{ n }{ n_2 } } \frac{1}{\sqrt{n_2}} \sum_{j=1}^{n_2} (X_{2j} - \mu ) \\
  & = \sum_{j=1}^{n} \zeta_{ni} + o_P(1),
\end{align*}
where
\[
  \zeta_{ni} = \left\{
    \begin{array}{ll}
      \frac{1}{\sqrt{n}} \rho_1 \lambda_1^{-1}  ( X_{1i} - \mu ), & \qquad i = 1, \dots, n_1, \\
      \frac{1}{\sqrt{n}} \rho_2 \lambda_2^{-1}  ( X_{2,i-n_1} - \mu ), & \qquad i  = n_1+1, \dots, n,
    \end{array}
  \right.
\]
for $ i = 1, \dots, n $, since $ n_i^{-1/2} \sum_{j=1}^{n_i} (X_{ij} - \mu ) = O_P(1) $ and
$ | \sqrt{n/n_i} - \sqrt{\lambda_i^{-1}} | = o(1) $, as $ n \to \infty $, such that
\begin{align*}
 \left| T_n(\rho_1, \rho_2) - \sum_{j=1}^{n} \zeta_{ni} \right| 
 & \le \sum_{i=1,2} | \rho_i | 
   \left| 
     \sqrt{ \frac{n}{n_i} } - \sqrt{ \lambda_i } 
   \right| 
   \left| 
     \frac{1}{ \sqrt{n_i} } \sum_{j=1}^{n_i} ( X_{ij} - \mu ) 
   \right| \\
   & = o_P(1),
\end{align*}
as $ n \to \infty $. Observe that, for each $ n \in \N $, the random variables $ \zeta_{n1}, \dots, \zeta_{nn} $ are mean zero and independent with $ \sup_{n \ge 1} \max_{1 \le i \le n} \Var( \zeta_{ni} ) < \infty $. Further,
\begin{align*}
  \sum_{i=1}^n \Var( \zeta_{ni} ) & =  \rho_1^2 \frac{n_1}{n} \lambda_1^{-2} \sigma_1^2 +   \rho_1^2 \frac{n_2}{n} \lambda_2^{-2} \sigma_2^2 \\
  & \to \rho_1^2 \lambda_1^{-1} \sigma_1^2 + \rho_2^2 \lambda_2^{-1} \sigma_2^2,
\end{align*}
as $ n \to \infty $. Hence, the Lindeberg condition is satisfied, such that the CLT implies
\[
  T_n(\rho_1, \rho_2) \stackrel{d}{\to} N( 0, \rho_1^2 \lambda_1^{-1} \sigma_1^2 + \rho_2^2 \lambda_2^{-1} \sigma_2^2 ),
\]
as $ n \to \infty $. This verifies (\ref{CLT_Vn}). Now the assertions follows by an application of Slutzky's lemma:
\[
  \sqrt{n}( \wh{\mu}_n(\gamma) - \mu ) = b_n( \wh{\theta}_n )' V_n \stackrel{d}{\to} b(\theta)' Z \sim N(0, \gamma^\le(\theta)^2 \lambda_1^{-1} \sigma_1^2 + (1-\gamma^\le(\theta))^2 \lambda_2^{-1} \sigma_2^2  ),
\]
as $ n \to \infty $, where $ \sigma^2( \gamma ) $ is given in (\ref{AsVar}). 
\end{proof}

\begin{remark}
Under the assumptions of Theorem~\ref{Theorem_Taylor} the CLT follows
directly from the asymptotic representation as a linear statistic shown there. 
\end{remark}

The CLT suggest the following estimator for the variance of the common mean
estimator $ \wh{\mu}_n(\gamma) $:
\begin{equation}
\label{CLTVarEst}
  \wh{\Var}(\wh{\mu}_n(\gamma)) = \frac{1}{n} \left[ \wh{\gamma}_n^2 \frac{n}{n_1} S_1^2 + (1-\wh{\gamma}_n )^2  \frac{n}{n_2} S_2^2 \right], 
\end{equation}
where $ \wh{\gamma}_n = \gamma^{\le}(n_1/n, n_2/n, S_1^2, S_2^2 ) $, if
$ \wt{S}_1^2 \le \wt{S}_2^2 $, and $ = \gamma^{>}(n_1/n, n_2/n, S_1^2, S_2^2 ) $ otherwise. The alternative formula
\[
  \wh{\Var}(\wh{\mu}_n(\gamma)) = \wh{\gamma}_n^2   \frac{S_1^2}{n_1} + (1-\wh{\gamma}_n )^2  \frac{S_2^2}{n_2} 
\]
shows how the standard errors $ S_i^2/n_i $,  $ i = 1, 2 $, of the samples
are weighted with the sample fractions $ n_i/n $, $ i = 1, 2 $, and the squared convex weights.

\section{Simulations and Data Analysis}
\label{sec: data_analysis}

We analyzed three data sets to illustrate both the method and its wide applicability. The data come from the fieds of natural science (physics), technology (quality engineering) and social information.

\subsection{Simulation study}

The simulation study aims at comparing the accuracy of the jackknife with the accuracy of the bootstrap taking into account the computational costs. Bootstrapping is a commonly used tool 
Recall that the nonparametric bootstrap draws $B$ random samples of sizes  $ n_1 $ and $ n_2 $ from the given data sets and then estimates the variance of a test statistics, the GD estimator in our study, by the sample variance of the $B$ replicates of the statistic. We consider the balanced design where $ N = n_1 = n_2 $. In this case the computational costs of the bootstrap, measured as the number of times the test statistic has to be evaluated, are equal to the costs of the jackknife, if $ B = N $. In our study we investigate the cases $ N = 25, 50, 75 $ and $ B = 100, 200, \dots, 1000 $, such that the computational costs of the bootstrap relative to those of the jackknife range from a factor of $4/3$ to $ 40 $.

The simulations consider normally distributed data (model 1), the $ t(5) $--distribution (model 2) as a distribution with fat tails, the $ U(-5,5) $-distribution as a short-tailed law and $\Gamma(a,b) $-distributed samples with mean $ a b $ and variance $ a b^2 $ (models 4 and 5) leading to skewed distributions. For model 4 observations distributed as $ \mu + \sigma_i ( \Gamma(a,\sigma_i) - a \sigma_i) $, $ i = 1, 2 $, were simulated with $ a = 1.5 $. Model 5 uses $ a = 2.5 $. The common mean equals $ \mu = 10 $.
The results are provided in Table~\ref{Sims}, which shows the coverage probabilities of the confidence intervals calculated based on the corresponding variance estimate for a nominal confidence level of $ 95\% $. Each entry is based on $ S = 20,000 $ runs.
 
 It can be seen that in all cases studied here the accuracy of the jackknife condifence interval in terms of the coverage probability is substantially higher than for the bootstrap intervals, which suffer from a lower than nominal coverage. Although the improvement, of course, diminishes for larger values of $B$, the bootstrap intervals are even worse for $ B = 1000 $. The results therefore demonstrate that jackknifing is a highly efficient tool when it comes to calculating confidence intervals on a large scale where the computational costs matter, as it is the case when analyzing big data.

\begin{table}
	\centering
	\begin{tabular}{rrrrrrrrrrrrr}
		\hline
		Model & N & Jack & 100 & 200 & 300 & 400 & 500 & 600 & 700 & 800 & 900 & 1000 \\ 
		\hline
		 1 &   25 & 0.946 & 0.935 & 0.938 & 0.932 & 0.937 & 0.933 & 0.938 & 0.936 & 0.938 & 0.940 & 0.937 \\ 
		 1 &   50 & 0.951 & 0.941 & 0.942 & 0.943 & 0.944 & 0.943 & 0.946 & 0.946 & 0.945 & 0.944 & 0.941 \\ 
		 1 &   75 & 0.951 & 0.942 & 0.943 & 0.945 & 0.947 & 0.945 & 0.946 & 0.945 & 0.946 & 0.945 & 0.944 \\ 
		 2 &   25 & 0.949 & 0.934 & 0.936 & 0.934 & 0.935 & 0.937 & 0.938 & 0.937 & 0.935 & 0.939 & 0.935 \\ 
		 2 &   50 & 0.949 & 0.941 & 0.942 & 0.941 & 0.940 & 0.943 & 0.943 & 0.942 & 0.944 & 0.942 & 0.941 \\ 
		 2 &   75 & 0.949 & 0.941 & 0.945 & 0.943 & 0.946 & 0.943 & 0.944 & 0.946 & 0.944 & 0.943 & 0.946 \\ 
		 3 &   25 & 0.946 & 0.940 & 0.940 & 0.940 & 0.941 & 0.937 & 0.937 & 0.940 & 0.941 & 0.943 & 0.942 \\ 
		 3 &   50 & 0.948 & 0.941 & 0.944 & 0.946 & 0.947 & 0.944 & 0.945 & 0.943 & 0.946 & 0.945 & 0.944 \\ 
		 3 &   75 & 0.948 & 0.946 & 0.947 & 0.944 & 0.947 & 0.948 & 0.949 & 0.950 & 0.949 & 0.948 & 0.946 \\ 
 4 &   25 & 0.921 & 0.908 & 0.915 & 0.912 & 0.911 & 0.910 & 0.913 & 0.916 & 0.912 & 0.913 & 0.912 \\ 
 4 &   50 & 0.933 & 0.928 & 0.926 & 0.930 & 0.924 & 0.931 & 0.931 & 0.931 & 0.926 & 0.928 & 0.928 \\ 
 4 &   75 & 0.940 & 0.932 & 0.934 & 0.933 & 0.937 & 0.937 & 0.935 & 0.937 & 0.933 & 0.934 & 0.934 \\ 
 5 &   25 & 0.930 & 0.918 & 0.919 & 0.920 & 0.918 & 0.919 & 0.917 & 0.922 & 0.923 & 0.920 & 0.920 \\ 
 5 &   50 & 0.938 & 0.932 & 0.929 & 0.933 & 0.936 & 0.931 & 0.934 & 0.935 & 0.932 & 0.932 & 0.935 \\ 
 5 &   75 & 0.942 & 0.935 & 0.936 & 0.938 & 0.940 & 0.938 & 0.939 & 0.936 & 0.937 & 0.939 & 0.939 \\ 		\hline
	\end{tabular}
\caption{Accurary of the jackknife and the bootstrap variance estimators in terms of the coverage probability for the confidence level $ 0.95 $.}
\label{Sims}
\end{table}

\subsection{Physics: Acceleration due to gravity}

In physics the common mean model frequently applies when observing or measuring a time-invariant physical phenomon. As an example, we analyze the Heyl and Cook measurements of the acceleration due to gravity,
\cite{HC36}.
Two of those data sets, taken from \cite{Cressie97}, are given by $ x_1 = (78, 78, 78, 86, 87, 81, 73, 67, 75, 82, 83)' $
and $ x_2 = (84, 86, 85, 82, 77, 76, 80, 83, 81, 78, 78, 78) $. 

These measurements are
	deviations from the value $ g = 980,060 \cdot 10^3 cm/sec^2 $. The classical $F$-test for homogeneity of variances accepts the null hypothesis of equal variances. The common mean assumption is obviously satisfied, since the same physical (invariant) phenomen is measured; systematic errors can be excluded from considerations, due to the great amount of care taken in the experiments to avoid systematic errors. 
The Graybill-Deal estimator for this data is $ \wh{\mu}^{(GD)} = 80.26123 $.
In order to estimate its variance and set up confidence intervals valid also under non-normal underlying distributions, we used the asymptotic approach based on the CLT, i.e. (\ref{CLTVarEst}), and the jackknife for unequal sample sizes. Table~\ref{TabPhysicsGD} provides those  estimates and the resulting confidence intervals for a confidence level of $ 95\% $ for the common mean based on the central limit theorem. 

\begin{table}
\begin{center}
	\begin{tabular}{lclr} \hline
	Method &	Est. sd of $ \wh{\mu}^{(GD)} $ & Confidence Interval & Width \\ \hline
	Asymptotics & $ 0.8455307$ & $ [  78.60399 ,  81.91847  ] $ & $ 3.31448 $\\
	Jackknife & $0.8492987$ & $ [  78.5966 ,  81.92585  ] $ & $  3.329251$ 
    \\ \hline
	\end{tabular}
\end{center}
\caption{Common mean estimation using the GD estimator: Standard errors and confidence intervals using asymptotics and the jackknife, respectively.}
\label{TabPhysicsGD}
\end{table}

It can be seen that the jackknife and the asymptotic formula lead to quite similar results for the GD estimator. Table~\ref{TabPhysicsNair} showsthe results for Nair's estimator:

\begin{table}
\begin{center}
	\begin{tabular}{lclr} \hline
		Method &	Est. sd of $ \wh{\mu}^{(N)} $ & Confidence Interval & Width \\ \hline
		Asymptotics & $ 1.279913 $ & $[  77.31746 ,  81.91847  ] $ & $ 5.017259 $\\
		Jackknife & $0.9752919$ & $ [  77.91451 ,  81.92585  ] $ & $  3.823144 $ 
		\\ \hline
	\end{tabular}
\end{center}
\caption{Common mean estimation using Nair's estimator:  Standard errors and confidence intervals using asymptotics and the jackknife, respectively.}
\label{TabPhysicsNair}
\end{table}

Nair's estimator uses the weights $ n_1/n = 11/23$ and $ n_2/n =12/23 $, since the variance estimates $ s_1^2 = 34.09091 $ and $ s_2^2 = 11.15152 $ are not ordered, which substantially from the stochastic weights used by the GD estimator. The jackknife variance estimate provides a tighter confidence  interval than the asymptotic approach in this case.

\subsection{Technology: Chip production data}

Chips as used in electronic circuits are cut using a cutting machine. The critical quantity is the width of cut out chips.
Figure~\ref{Figure1} depicts kernel density estimates (with cross-validated bandwidth selection) and boxplots of two series of $ n = m = 240$ of width measurements using different cutting saws. Both data sets are non-normal, as confirmed by  Shapiro-Wilk tests leading to p-values $ < 10^{-3} $, and there are some outliers present which are, however, not treated for the present analysis.
The sample means  are $ \overline{x} = 6.293254 $ and $ \overline{y} = 6.292667 $ with estimated standard deviations $ s_x = 0.003785844 $ and $ s_y = 0.004962341 $. 
Table~\ref{TabChipsNair} provides the results when assuming $ \sigma_1^2 < \sigma_2^2 $.

\begin{table}
\begin{center}
	\begin{tabular}{lclr} \hline
		Method &	Est. sd of $ \wh{\mu}^{(N)} $ & Confidence Interval & Width \\ \hline
		Asymptotics & $ 0.0001942891 $ & $ [  6.292657 ,  6.293419  ] $ & $ 0.0007616134 $\\
		Jackknife & $ 0.0001941049 $ & $ [  6.292658 ,  6.293418  ] $ & $  0.000760891 $ 
		\\ \hline
	\end{tabular}
\end{center}
\caption{Nair's estimator for the chip data:  Standard errors and confidence intervals using asymptotics and the jackknife, respectively.}
\label{TabChipsNair}
\end{table}

When interchanging the samples, such that the variance estimates are not ordered, the results for Nair's estimator differ, see Table~\ref{TabChipsNairChanged}. Now the jackknife results in a substantially smaller variance estimate and a tighter confidence interval compared to the approach based on the asymptotics.

\begin{table}
\begin{center}
	\begin{tabular}{lclr} \hline
		Method &	Est. sd of $ \wh{\mu}^{(N)} $ & Confidence Interval & Width \\ \hline
		Asymptotics & $ 0.0002921816 $ & $ [  6.292388 ,  6.293419  ] $ & $ 0.001145352  $\\
		Jackknife & $ 0.0001984456 $ & $ [  6.292571 ,  6.293418  ] $ & $  0.0007779067 $ 
		\\ \hline
	\end{tabular}
\end{center}
\caption{Results for Nair's estimator for the chip data when switching the samples.}
\label{TabChipsNairChanged}
\end{table}

\begin{center}
	\begin{figure}
			\includegraphics[width=7cm]{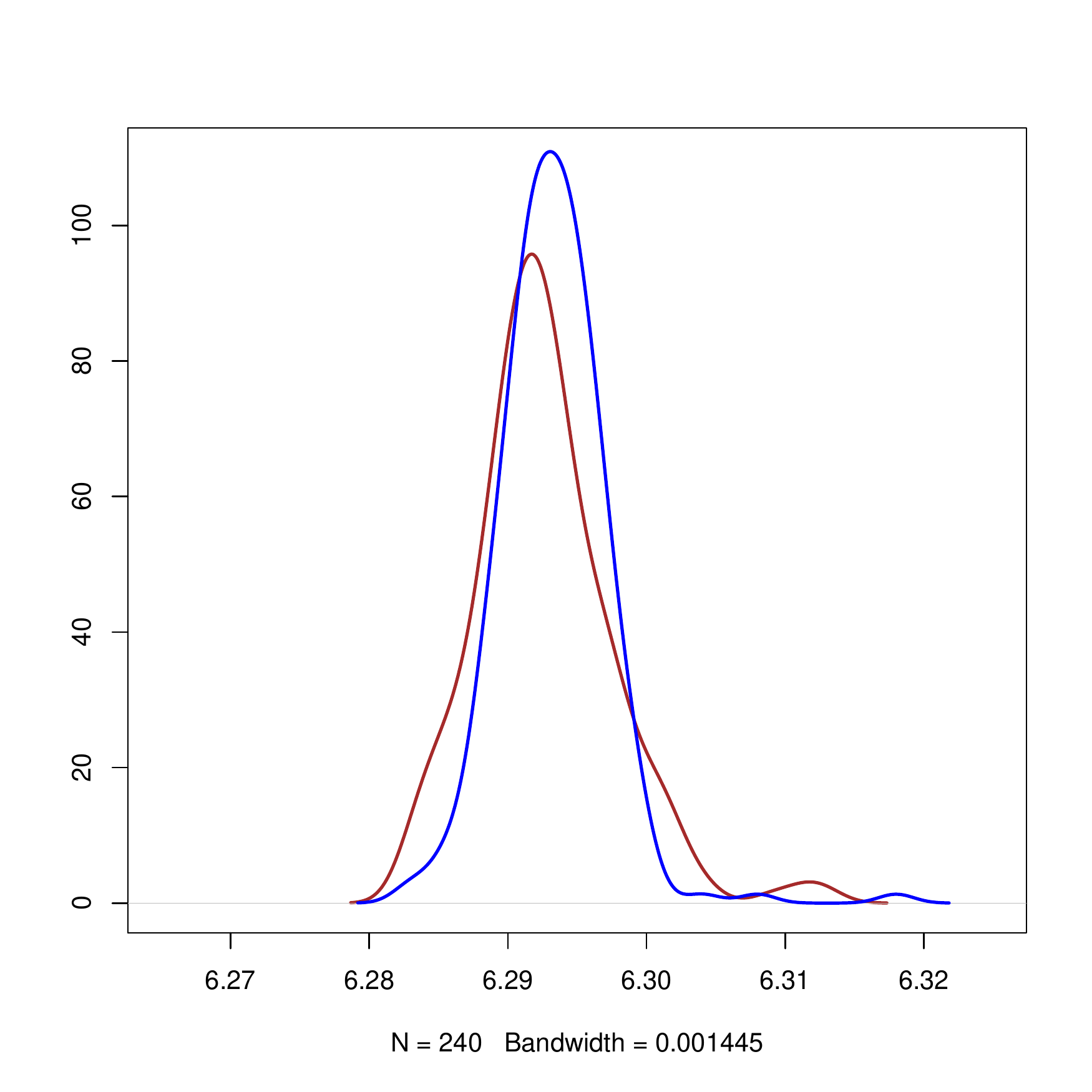}
			\includegraphics[width=7cm]{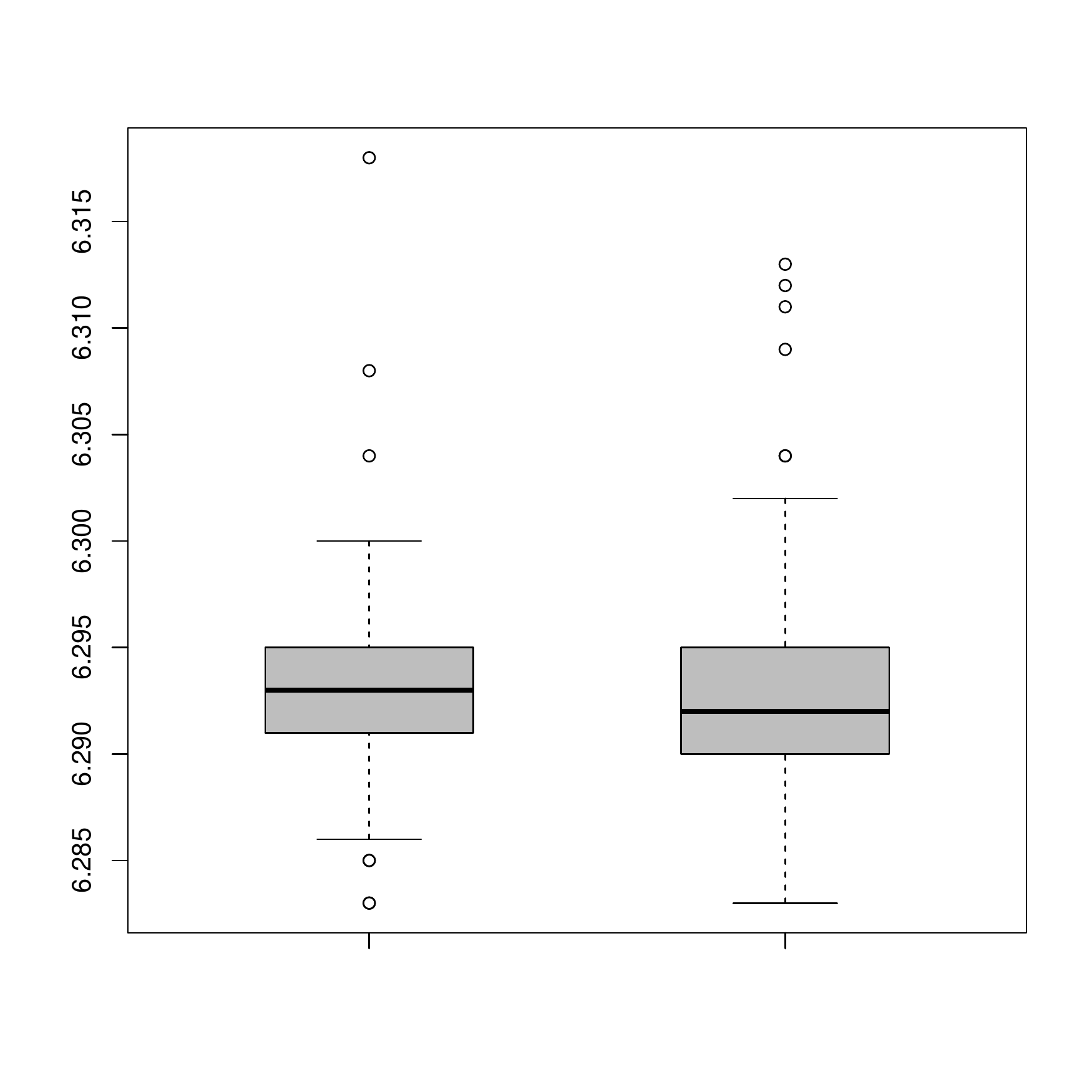}
	\caption{Density estimates and boxplots of width measurements of sliced chips using two cutting saws.}
	\label{Figure1}
	\end{figure}
\end{center}

\subsection{Social Information: Japanese child data}

In Japan the strength of eight years old boys and girls was investigated in the six prefectures (Aomori, Iwate, Miyagi, Akita, Yamagata and Fukushima) 
of the Touhoku region on Honshu, the largest island of Japan, and the prefecture Hokkaido located in the north of Japan. The observations for the boys are $ 52.55, 54.08, 54.25, 52.92, 56.31, 53.63, 52.52 $ and those for girls $ 52.95, 55.72, 56.14, 54.24, 58.19, 55.32, 54.45 $. 
For those data sets the Graybill-Deal estimator is given by $ 54.34878 $. 

Table~\ref{TabChildGD} provides the results when assigning the girls' data to $x_1 $ and the boys' data to $ x_2 $. The results for the GD estimator are as follows.

\begin{table}
\begin{center}
	\begin{tabular}{lclr} \hline
		Method &	Est. sd of $ \wh{\mu}^{(GD)} $ & Confidence Interval & Width \\ \hline
		Asymptotics & $ 0.3921168 $ & $[  53.58023 ,  55.11733  ] $ & $ 1.537098$\\
		Jackknife & $0.6874476 $ & $ [  53.00139 ,  55.69618  ] $ & $  2.694794 $ 
		\\ \hline
	\end{tabular}
\end{center}
\caption{Results for the GD common estimator for the Japanese child data: Standard errors and confidence intervals using asymptotics and the jackknife.}
\label{TabChildGD}
\end{table}

For Nair's estimator we obtain different results, since the variance estimates are not ordered, see Table~\ref{TabChildNair}. As for the chip data above, the jackknife leads to a smaller variance estimate and a tighter confidence interval.

\begin{table}
\begin{center}
	\begin{tabular}{lclr} \hline
		Method &	Est. sd of $ \wh{\mu}^{(N)} $ & Confidence Interval & Width \\ \hline
		Asymptotics & $  0.5919936 $ & $[  53.35898 ,  55.11733  ]$ & $ 2.320615$\\
	Jackknife & $0.5593932 $ & $ [  53.42288 ,  55.69618  ]$ & $  2.192821  $ 
	\\ \hline
\end{tabular}
\end{center}
\caption{The Nair estimator for Japanese child data: Standard errors and confidence intervals using asymptotics and the jackknife.}
\label{TabChildNair}
\end{table}

\section*{Acknowledgments}

Parts of this paper have been written during research visits of the first author at Mejiro University, Tokyo. He thanks for the warm hospitality. Both authors thank Hideo Suzuki, Keio University at Yokohama, for invitations to his research seminar and the participants for comments and discussion. The authors gratefully acknowlegde the support of Prof. Takenori Takahashi, Mejiro University and Keio University Graduate School, and Akira Ogawa, Mejiro University, for providing and discussing the chip manufacturing data. They thank anonymous referees for helpful comments.

\appendix

\section{Miscellaneous results}

\begin{lemma} 
\label{SampleVarAsLinear}
Let $ \xi_1, \dots, \xi_n $ be i.i.d. random variables \underline{with $ E |\xi_1|^{12} < \infty $.} Then the sample variance $ S_n^2 = \frac{1}{n}\sum_{i=1}^n (\xi_i - \overline{\xi})^2 $ is asymptotically linear,
\begin{equation}
\label{SampleVarianceAsLinear}
  S_n^2 - \sigma^2 = \frac{1}{n} \sum_{i=1}^n (\wt{\xi}_i^2 - \sigma^2) - R_n,
\end{equation}
where $ \wt{\xi}_i = \xi_i - E( \xi_i ) $, $ 1 \le i \le n $, and the remainder term 
$ R_n = ( \frac{1}{n} \sum_{i=1}^n (X_i-\mu))^2 $ satisfies
\begin{equation}
\label{PropRemainderSampleVar}
  E ( n R_n^6 ) = O(n^{-5}), \quad E( n R_n^4 ) = O( n^{-3} ), \quad E( n R_n^2 ) = O(n^{-1}),
\end{equation}
and
\begin{equation}
\label{HigherOrderApproxRemainder}
  E( R_n - R_{n-1} )^2 = O( n^{-3} ), \ E( R_n - R_{n-1} )^4 = O(n^{-6} ),
\end{equation}
as $ n \to \infty $. The assertions also hold for the unbiased variance estimator $ \wt{S}_n^2 = \frac{n}{n-1} S_n^2 $.
\end{lemma}

\begin{proof}[Proof of Lemma~\ref{SampleVarAsLinear}]
A direct calculation shows that
\begin{align*}
  S_n^2 - \sigma^2 &=
  \frac{1}{n} \sum_{i=1}^n ( \wt{\xi}_i^2 - \sigma^2 ) - ( \overline{\wt{\xi}} )^2, 
\end{align*}
with $ R_n = ( \overline{\wt{\xi}} )^2 $. Using the fact that 
$ E( \sum_{i=1}^n \wt{\xi}_i )^r = O( n^{r/2} ) $, if $ E | \xi_i |^r < \infty $, the
estimates in (\ref{PropRemainderSampleVar}) follow easily, i.e. 
\[
  E( n R_n^2 ) = n n^{-4} E\left( \sum_{i=1}^n \wt{\xi}_i \right)^{4}  = O(n^{-1}),
\]
\[
  E( n R_n^4 ) = O( n^{-3} )
\]
and
\[
  E( n R_n^6 ) = n n^{-12} E\left( \sum_{i=1}^n \wt{\xi}_i \right)^{12} = O( n^{-5} ).
\]
To show (\ref{HigherOrderApproxRemainder}) observe that
\begin{align*}
  R_n - R_{n-1} &= \left( \frac1n \sum_{j=1}^n \wt{\xi}_j  \right)^2 -
  \left( \frac{1}{n-1} \sum_{j=1}^{n-1} \wt{\xi}_j  \right)^2 \\
   & =\left( \frac{1}{n^2} - \frac{1}{(n-1)^2} \right) \sum_{j,j'<n} \wt{\xi}_j \wt{\xi}_{j'} + \frac{2}{n^2} \sum_{j=1}^n \wt{\xi}_n \wt{\xi}_j.
\end{align*}
By virtue of the $ C_r $-inequality,
\[
  E( R_n - R_{n-1} )^r \le 2^{r-1}  \left\{ E    \left( 
  \sum_{j,j'=1}^{n-1} \frac{-2n+1}{n^2(n-1)^2} \wt{\xi}_j \wt{\xi}_{j'}
  \right)^r
  +
  E \left(
    \frac{2}{n^2} \sum_{j=1}^n \wt{\xi}_n \wt{\xi}_j
  \right)^r \right\},
\]
for $ r = 2, 4 $. Observe that
\begin{align*}
  E \left(
  \frac{2}{n^2} \sum_{i=1}^n \wt{\xi}_n \wt{\xi}_j
  \right)^2
  & \le \frac{4}{n^4} \sqrt{ E \wt{\xi}_n^4 } \sqrt{ E \left( \sum_{j=1}^n \wt{\xi}_j \right)^4 } 
  = O( n^{-3} )
\end{align*}
and
\begin{align*}
  E \left(
   \sum_{j,j'=1}^{n-1} \frac{-2n+1}{n^2(n-1)^2} \wt{\xi}_j \wt{\xi}_{j'}    
  \right)^2
  & = \left( \frac{-2n+1}{n^2(n-1)^2}   \right)^2 E \left(\sum_{j=1}^{n-1} \wt{\xi}_j \right)^4
  = O( n^{-4} ),
\end{align*}
such that
\[
  E( R_n - R_{n-1})^2 = O( n^{-3} ).
\]
Similarly,
\begin{align*}
E \left(
\frac{2}{n^2} \sum_{j=1}^n \wt{\xi}_n \wt{\xi}_j
\right)^4
& \le \frac{2^4}{n^8} \sqrt{ E \wt{\xi}_n^8 } \sqrt{ E \left( \sum_{j=1}^n \wt{\xi}_j \right)^8 } 
= O( n^{-8} ) O( n^2 ) = O(n^{-6})
\end{align*}
and
\begin{align*}
  E \left(  
   \sum_{j,j'=1}^{n-1} \frac{-2n+1}{n^2(n-1)^2} \wt{\xi}_j \wt{\xi}_{j'}
   \right)^4
   & = \left( \frac{-2n+1}{n^2(n-1)^2}   \right)^4 E \left(\sum_{j=1}^{n-1} \wt{\xi}_j \right)^8 \\
   & = O( n^{-12} ) O( n^4 ) = O(n^{-8})
\end{align*}
leading to
\[
  E(R_n - R_{n-1})^4 = O(n^{-6} ).
\]
The additional arguments to treat $ \wt{S}_n^2 $ are straightforward and omitted.
\end{proof}

\begin{lemma}
\label{EstimationVariances}
	Suppose that $ X_{ij}, j = 1, \dots, n_i $, are i.i.d. with finite eight moments and strictly ordered variances $ \sigma_1^2 < \sigma_2^2 $. Then
	\[
	P( S_1^2 > S_2^2 ) \le C n^{-2}, 
	\]
	for some constant $ C $.
\end{lemma}

\begin{proof}
	By Lemma~\ref{SampleVarAsLinear} we have
	\[
	S_{i}^2 - \sigma_i^2 = L_{ni} - R_{ni}, 
	\]
	with $ L_{ni} = \frac{1}{n_i} \sum_{j=1}^{n_i} [ ( X_{ij}- \mu)^2 - \sigma_i^2 ] $
	and $ R_{ni} = ( \frac{1}{n_i} \sum_{j=1}^{n_i} (X_{ij}-\mu) )^2 $, for $ i = 1, 2 $.  For $ r = 1, 2 $
	we have
	\[
	E | L_{ni} |^{2r} = O( n_i^{-r} ) \qquad \text{and} \qquad E | R_{ni} |^{2r} = O(n_i^{-2r}).
	\]
	Hence
	\[
	\| S_1^2 - S_2^2 - (\sigma_1^2-\sigma_2^2) \|_{L_4} \le 
	\| S_1^2 - \sigma_1^2 \|_{L_4} + \| S_2^2 - \sigma_2^2 \|_{L_4} 
	= O( n_1^{-1/2} ) + O( n_2^{-1/2} ) 
	= O( n^{-1/2} ).
	\] 
	Now we may conclude that
	\begin{align*}
	P( S_1^2 > S_2^2 ) & \le P( |S_1^2 - S_2^2 - (\sigma_2^2 - \sigma_1^2) | \ge (\sigma_2^2 - \sigma_1^2)  ) \\
	& \le \frac{ E | S_1^2 - \sigma_1^2 - (S_2^2 - \sigma_2^2) |^4  }{ | \sigma_2^2 - \sigma_1^2 |^4  } \\
	& \le \frac{ ( \| S_1^2 - \sigma_1^2 \|_{L_4} + \| S_2^2 - \sigma_2^2 \|_{L_4} )^4 }{  | \sigma_2^2 - \sigma_1^2 |^4 } \\
	& = O( n^{-2} ).
	\end{align*}
\end{proof}

\begin{lemma}
	\label{AppLemma1}	
	If $ \xi_1, \dots, \xi_{n_1} $ are i.i.d. with finite second moment and $ \max_{1 \le i \le n_1} E| \wh{\xi}_i - \xi_i | = o(1) $,
	then
	\[
	E \left| \frac{1}{n_1} \sum_{j=1}^{n_1} \wh{\xi}_j^2 - \frac{1}{n_1} \sum_{j=1}^{n_1} \xi_j^2 \right|  = o(1),
	\]
	and
	\[
	E \left| \left( \frac{1}{n_1} \sum_{j=1}^{n_1} \wh{\xi}_j \right)^2 - 
	\left( \frac{1}{n_1} \sum_{j=1}^{n_1} \xi_j \right)^2 \right| = o(1),
	\]
	as $ n_1 \to \infty $.
\end{lemma}
\begin{proof}
	Since
	\[
	E \left| \frac{1}{n_1} \sum_{j=1}^{n_1} \wh{\xi}_j^2 - \frac{1}{n_1} \sum_{j=1}^{n_1} \xi_j^2 \right| \le \frac{1}{n_1} \sum_{j=1}^{n_1} E|\wh{\xi}_j^2 - \xi_i^2| = o(1),
	\]
	as $ n_1 \to \infty $. Similarly,
	\[
	E | \wh{\xi}_i \wh{\xi}_j - \xi_i \xi_j | = E| (\xi_i + A_n )(\xi_j + A_n)  - \xi_i \xi_j  |
	\le E|\xi_i A_n| + E |\xi_j A_n| + E( A_n^2 ) = o(1),
	\] 
	such that the squared sample moments converge in $ L_1 $ as well, since
	\[
	E \left| \left( \frac{1}{n_1} \sum_{j=1}^{n_1} \wh{\xi}_j \right)^2 - 
	\left( \frac{1}{n_1} \sum_{j=1}^{n_1} \xi_j \right)^2 \right|
	\le \frac{1}{n_1^2} \sum_{i,j=1}^{n_1} E | \wh{\xi}_i \wh{\xi}_j - \xi_i \xi_j |  = o(1),
	\]
	as $ n_1 \to \infty $, follows. 
\end{proof}

\begin{lemma}
	\label{Conv2ndMoments}
	Suppose that $ Z_1, Z_2, \dots $ are i.i.d. with $ E( Z_1^2 )< \infty $ and $ \xi_1, \xi_2, \dots $ are random variables with 
	$ \max_{1 \le i \le n} | Z_i - \xi_i | = o(1) $, as $ n \to \infty $, a.s.. Then
	$ \frac{1}{n} \sum_{i=1}^n \xi_i \to E(Z_1) $, a.s.,  and $ \frac{1}{n} \sum_{i=1}^n \xi_i^2 \to E( Z_1^2 ) $, a.s.
\end{lemma}

\begin{proof}
	First notice that by the strong  law of large numbers
	\begin{align*}
	\left| \frac{1}{n} \sum_{i=1}^n \xi_i  - E(Z_1) \right|  
	& = \left|  \frac{1}{n} \sum_{i=1}^n ( Z_i - E(Z_1) )+  \frac{1}{n} \sum_{i=1}^n  (\xi_i - Z_i)  \right| \\
	& \le  \left|  \frac{1}{n} \sum_{i=1}^n ( Z_i - E(Z_1)  ) \right| + \max_{1 \le i \le n} | Z_i - \xi_i |, \\
	& = o(1),
	\end{align*}
	as $ n \to \infty $, a.s.. For the sample moments of order $2$, we have the estimate
	\begin{align*}
	\left| \frac{1}{n} \sum_{i=1}^n Z_i^2 - \frac{1}{n} \sum_{i=1}^n \xi_i^2 \right| 
	& \le  \left| \frac{1}{n} \sum_{i=1}^n | Z_i -  \xi_i | (|Z_i| + |\xi_i|) \right| \\
	& \le \max_{1 \le i \le n} | Z_i - \xi_i | \left(
	\frac{1}{n} \sum_{i=1}^n | Z_i | +  \frac{1}{n} \sum_{i=1}^n | \xi_i |  \right) \\
	& = \max_{1 \le i \le n} | Z_i - \xi_i | ( 2 E | Z_1 | + o(1) ) \\
	& = o(1),
	\end{align*}
	as $ n \to \infty $.
\end{proof}

\section{Proof of Theorem~\ref{Theorem_Taylor}}
\label{AppendixB}

In the sequel, $ \| \cdot \|_2 $ denotes the vector-2 norm, whereas $ \| \cdot \|_{L_p} $ stands for the $ L_p $-norm of a random variable, i.e. $ \| Z \|_{L_p} = \left( E |Z|^p \right)^{1/p} $.  

As discussed in  Remark~\ref{RemarkOneSample}, we let
\begin{equation}
\label{DefParVector}
\wh{\theta}_N = (N/n, \wt{S}_1^2, \wt{S}_2^2, \overline{X}_1, \overline{X}_2 )', 
\qquad \theta = (\lambda_1, \sigma_1^2, \sigma_2^2, \mu_1, \mu_2 )'.
\end{equation}
Also notice that $ \mu_1 = \mu_2 = \mu $. We have
\begin{equation}
\label{Expres}
\varphi( \wh{\theta}_N )  := \wh{\mu}_N 
= \gamma_N \overline{X}_1 + (1- \gamma_N) \overline{X}_2 
= 
\phi( \wh{\theta}_N ) + \psi( \wh{\theta}_N ),
\end{equation}
where
\begin{align*}
\phi( \wh{\theta}_N )  & = \gamma^\le( \wh{\theta}_N) \overline{X}_1 \eins_{\{ \wt{S}_1^2 \le \wt{S}_2^2 \}} + \gamma^>( \wh{\theta}_N) \overline{X}_1 \eins_{\{ \wt{S}_1^2 > \wt{S}_2^2 \}}, \\
\psi( \wh{\theta}_N ) & = 
[1-\gamma^\le(\wh{\theta}_N) ] \overline{X}_2 \eins_{\{ \wt{S}_1^2 \le \wt{S}_2^2 \}} +
[1-\gamma^>(\wh{\theta}_N) ] \overline{X}_2 \eins_{\{ \wt{S}_1^2 > \wt{S}_2^2 \}}.
\end{align*}
Observe that 
\begin{align*}
\phi( \theta ) & = \gamma^{\le}( \theta ) \mu_1 \eins_{ \{ \sigma_1^2 \le \sigma_2^2 \} }
+ \gamma^{>}( \theta ) \mu_1 \eins_{ \{ \sigma_1^2 > \sigma_2^2 \} }, \\
\psi( \theta ) & = [1-\gamma^{\le}(\theta)]\mu_2 \eins_{\{\sigma_1^2\le\sigma_2^2\}} +
[1-\gamma^{>}(\theta)] \mu_2 \eins_\{\sigma_2^2>\sigma_2^2\}, 
\end{align*}
such that especially the function $ \varphi $ satisfies
\begin{align*}
\varphi( \theta ) 
& = 
\gamma^\le(\theta) \mu_1 \eins_{\{ \sigma_1^2 \le \sigma_2^2 \}} +
\gamma^>(\theta) \mu_1 \eins_{\{ \sigma_1^2 > \sigma_2^2 \}} \\
& \quad
+ [1-\gamma^\le(\theta) ] \mu_2 \eins_{\{ \sigma_1^2 \le \sigma_2^2 \}} +
[1-\gamma^>(\theta) ] \mu_2 \eins_{\{ \sigma_1^2 > \sigma_2^2 \}}.
\end{align*}
Observe that 
\[
\varphi(  \theta ) = \mu, \qquad \text{if $ \mu_1 = \mu_2 = \mu $}.
\]
We show the result in detail for $ \theta $ with $ \sigma_1^2 < \sigma_2^2 $; the other case ($ \sigma_1^2 > \sigma_2^2 $) is treated analogously and we only indicate some essential steps. Further, we only discuss $ \phi( \wh{\theta}_N) $, since $ \psi( \wh{\theta}_N) $ can be handled similarly. We have to consider
\begin{align*}
\phi(\wh{\theta}_N ) - \phi( \theta )
& = \gamma^\le( \wh{\theta}_N ) \overline{X}_1 \eins_{\{ \wt{S}_1^2 \le \wt{S}_2^2\}}
+ \gamma^>( \wh{\theta}_N ) \overline{X}_1 \eins_{\{ \wt{S}_1^2 > \wt{S}_2^2 \}} \\
& \quad - \gamma^\le( \theta) \mu_1,
\end{align*}
which can be rewritten as
\begin{align*}
\phi( \wh{\theta}_N ) - \phi( \theta )
& =  [ \gamma^\le( \wh{\theta}_N ) \overline{X}_1 - \gamma^\le( \theta) \mu_1 ] \eins_{\{ \wt{S}_1^2 \le \wt{S}_2^2\}} \\
& \quad + [ \gamma^>( \wh{\theta}_N ) \overline{X}_1 - \gamma^>( \theta) \mu_2 ] \eins_{\{ \wt{S}_1^2 > \wt{S}_2^2 \}} \\
& \quad - \gamma^\le( \theta ) \mu_1 \eins_{\{ \wt{S}_1^2 > \wt{S}_2^2\}}
+  \gamma^>( \theta ) \mu_2 \eins_{\{ \wt{S}_1^2 > \wt{S}_2^2 \}} 
\end{align*}
The first term is the leading one, whereas the last two terms are of the order $ 1/N^2 $ in the $ L_2 $-sense, since e.g.
\[
E ( \gamma^\le( \theta ) \mu_1 \eins_{\{ \wt{S}_1^2 > \wt{S}_2^2 \}} )^2
\le \| \gamma^\le \|_\infty^2 \mu_1^2 P( \wt{S}_1^2 > \wt{S}_2^2 )
\]
and for $ \sigma_1^2 < \sigma_2^2 $ we have
\begin{equation}
\label{HigherOrderEstimateSampleVariance}
P( \wt{S}_1^2 \ge \wt{S}_2^2 ) = O\left( \frac{1}{N^2} \right)
\end{equation}
by Lemma~\ref{EstimationVariances}.
By symmetry, we also have
\begin{equation}
\label{HigherOrderEstimateSampleVariance2}
P( \wt{S}_1^2 < \wt{S}_2^2 ) = O(1/N^2), \qquad \text{if $ \sigma_1^2 > \sigma_2^2 $}.
\end{equation}
By our assumptions on $ \gamma^\le $, we have for the leading term of 
$ \phi( \wh{\theta}_N ) - \phi( \theta ) $
\begin{align*}
\gamma^\le( \wh{\theta}_N ) \overline{X}_1 - \gamma^\le( \theta )\mu_1
&= \nabla [\gamma^\le( \theta ) \mu_1] (\wh{\theta}_N - \theta) 
+ \frac12 (\wh{\theta}_N - \theta)' \Delta [ \gamma^\le(\theta) \mu_1]( \wh{\theta}_N^* ) ( \wh{\theta}_N - \theta ),
\end{align*}
which can be written as
\begin{equation}
\label{TaylorExpansion}
\sum_{i=1}^5 \frac{ \partial [\gamma^\le(\theta)\mu_1] }{ \partial \theta_i } (\wh{\theta}_{Ni} - \theta_i) + \frac12 \sum_{i,j=1}^5 \partial_{ij} [\gamma^\le(\theta) \mu_1]( \wh{\theta}_N^* ) ( \wh{\theta}_{Ni} - \theta_i )( \wh{\theta}_{Nj} - \theta_j ),
\end{equation}
for some $ \wh{\theta}_N^* $ between $ \wh{\theta}_N $ and $ \theta $, where
$ \wh{\theta}_N = ( \wh{\theta}_{Ni} )_{i=1}^5 $ 
and $ \theta = (\theta_i)_{i=1}^5 $. 
$\nabla [\gamma^\le( \theta ) \mu_1] $ denotes the gradient of the function $ \bar{\theta} \mapsto \gamma^\le(\bar{\theta}) \bar{\theta}_4 $, $ \bar{\theta} =(\bar{\theta}_1, \dots, \bar{\theta}_5)' \in \Theta $, evaluated at (the true value) $ \theta $,
and $ \partial_{ij} [\gamma^\le \mu_1]( \wh{\theta}_N^* ) $ stands for the second partial derivative with respect to $ \theta_i $ and $ \theta_j $ evaluated at a point $ \wh{\theta}_N^* $.
Observe that the linear term in the above expansion is asymptotically linear, i.e. it can be written as
\[
\nabla[ \gamma^{\le}( \theta ) \mu_1 ]( \wh{\theta}_N - \theta ) =
\frac1N \sum_{j=1}^N \left\{ \sum_{i=1}^5 \frac{\partial [\gamma^\le(\theta)\mu_1] }{ \partial \theta_i} \xi_{ij} \right\} + R_N^{\gamma^\le}, \quad N E( R_N^{\gamma^\le} )^2 = o(1),
\]
for appropriate random variables $ \xi_{ij} $, since the coordinates are asymptotically linear, i.e. 
\[ \wh{\theta}_{Ni} - \theta_i = \frac1N \sum_{j=1}^N \xi_{ij} + R_{Ni}, \quad \text{with} \quad  N E(R_{Ni}^2) = o(1). \]
This follows from the facts that
$ \overline{X}_i $, $ i = 1, 2 $, are linear statistics and $ \wt{S}_i^2 - \sigma_i^2 $, $ i = 1, 2 $, are asymptotically linear by virtue of Lemma~\ref{SampleVarAsLinear}. Thus, it
suffices to estimate the second order terms of the Taylor expansion.  First observe that   
\begin{equation}
\label{EstThetaHat}
N E( \wh{\theta}_{Ni} - \theta_i )^4 = N E \left| \frac1N \sum_{j=1}^N \xi_{ij} + R_{Ni} \right|^4 = O(1/N),
\end{equation}
for $ i = 1, \dots, 5 $, such that
\begin{equation}
\label{EstThetaHatSq}
E \| \wh{\theta}_N - \theta \|_2^2 = O(1/N^2),
\end{equation}
where $ \| \cdot \|_2 $ denotes the vector-2 norm. For the arithmetic means (\ref{EstThetaHat}) follows because the remainder term vanishes and $ E(\sum_{j=1}^N \xi_{ij} )^r = O(N^{r/2})$ if $ E|\xi_{ij}|^r < \infty$, and for the estimators $ \wt{S}_i^2 $ because of Lemma~\ref{SampleVarAsLinear}. The quadratic terms can be estimated as follows.
\begin{align*}
& N E \left( \partial_{ij} [\gamma^\le(\theta) \mu_1] ( \wh{\theta}_N^* ) ( \wh{\theta}_{Ni} - \theta_i )( \wh{\theta}_{Nj} - \theta_j ) \right)^2 \\
& \quad \le \| \partial_{ij} [\gamma^\le(\theta) \mu_1]\|_\infty^2 N E\left( |  \wh{\theta}_{Ni} - \theta_i  |^2 |  \wh{\theta}_{Nj} - \theta_j  |^2 \right) \\
& \quad =  O\left( \sqrt{ N E | \wh{\theta}_{Ni} - \theta_i  |^4 } \sqrt{ N E | \wh{\theta}_{Nj} - \theta_j  |^4 } \right) \\
& \quad = O( 1/N ).
\end{align*}
Here $ \| x \|_\infty = \max_{1\le i \le \ell} | x_i | $ denotes the $ l_\infty $-norm of a vector $ x = (x_1, \dots, x_\ell)' \in \R^{\ell } $. Therefore, we obtain the estimate
\[
N E \left( \frac12 \sum_{i,j=1}^5 \partial_{ij} [\gamma^\le(\theta) \mu_1]( \wh{\theta}_N^* ) ( \wh{\theta}_{Ni} - \theta_i )( \wh{\theta}_{Nj} - \theta_j ) \right)^2 = O( 1/ N )
\]
for the quadratic term of the above Taylor expansion. Switching back to the definition (\ref{DefParVector}) 
of $ \wh{\theta}_N $ and $ \theta $, the above arguments show that for ordered variances $ \sigma_1^2 < \sigma_2^2 $
\[
\phi( N/n, \wh{\theta}_N ) - \phi( \lambda_1, \theta ) 
= \nabla [\gamma^\le( \theta ) \mu_1] (\wh{\theta}_N - \theta) \eins_{\{ \wt{S}_1^2 \le \wt{S}_2^2 \}}  
+ \nabla [\gamma^>( \theta ) \mu_2] (\wh{\theta}_N - \theta) \eins_{\{ \wt{S}_1^2 > \wt{S}_2^2 \}} + R_N^\phi,
\]
for some remainder $ R_N^\phi $ with $ N E( R_N^\phi)^2 = O(1/N) $. The treatment
of $ \psi( \wh{\theta}_N ) - \psi(\theta) $ is similar and leads to
\[
\psi( N/n, \wh{\theta}_N ) - \psi( \lambda_1, \theta ) 
= \nabla [(1-\gamma^\le( \theta )) \mu_2] (\wh{\theta}_N - \theta) \eins_{\{ \wt{S}_1^2 \ge \wt{S}_2^2 \}}  
+ \nabla [(1-\gamma^>( \theta )) \mu_1] (\wh{\theta}_N - \theta) \eins_{\{ \wt{S}_1^2 < \wt{S}_2^2 \}} + R_N^\psi,
\]
for some remainder term $ R_N^\psi $ with $ N E( R_N^\psi)^2 = O(1/N) $. Putting things together and collecting terms leads to
\begin{align*}
\wh{\mu}_N(\gamma) - \mu 
& = \biggl[ (\nabla [ \gamma^\le(\theta) \mu_1] + \nabla [(1-\gamma^\le(\theta)) \mu_2 ])(\wh{\theta}_N - \theta ) \eins_{\{ \wt{S}_1^2 \le \wt{S}_2^2 \}} \\
& \quad +
(\nabla [\gamma^>(\theta) \mu_2] + \nabla [(1-\gamma^>(\theta)) \mu_1]) (\wh{\theta}_N - \theta ) \eins_{\{ \wt{S}_1^2 > \wt{S}_2^2 \}} \biggr] + R_N^\gamma,
\end{align*}
for a remainder $ R_N^\gamma $ with $ N E( R_N^\gamma )^2 = O(1/N) $.
Replacing the indicators $ \eins_{\{ \wt{S}_1^2 \le \wt{S}_2^2 \}} $ and $ \eins_{\{ \wt{S}_1^2 > \wt{S}_2^2 \}} $ by $ \eins_{\{ \sigma_1^2 \le \sigma_2^2 \}} (=1) $ and  $ \eins_{\{ \sigma_1^2 > \sigma_2^2 \}} (=0) $, respectively, adds additional correction terms
$ R_N^{(1)} $ and $ R_N^{(2)} $ with
$ N E( R_N^{(i)} )^2 = o(1) $, $ i = 1, 2 $. This can be seen as follows: We have for instance
\[
\nabla [\gamma^\le( \theta )\mu_1]  \eins_{\{ \wt{S}_1^2 \le \wt{S}_2^2 \}}  ( \wh{\theta}_N - \theta )
= \nabla [\gamma^\le( \theta ) \mu_1] \eins_{\{ \sigma_1^2 \le \sigma_2^2 \}}   ( \wh{\theta}_N - \theta ) + U_N,
\]  
where
\[
U_N = \nabla [\gamma^\le( \theta ) \mu_1] ( \eins_{\{ \wt{S}_1^2 \le \wt{S}_2^2 \}}  - \eins_{\{ \sigma_1^2 \le \sigma_2^2 \}}  ) ( \wh{\theta}_N - \theta ).
\]
Observe that for $ \sigma_1^2 < \sigma_2^2 $ we have $ J = | \eins_{ \{ \wt{S}_1^2 \le \wt{S}_2^2 \}} - \eins_{ \{ \sigma_1^2 
	\le \sigma_2^2 \} } | = \eins_{ \{ \wt{S}_1^2 > \wt{S}_2^2 \}} $. 
If we put $ C = \max_j \sup_\theta \left| \frac{ \partial [\gamma^\le(\theta) \mu_1 ] }{ \partial \theta_j } \right| < \infty $, we may estimate
\begin{align*}
E(U_N^2) & = 
E \left( \nabla[ \gamma^\le( \theta ) \mu_1 ] (\wh{\theta}_N - \theta ) \eins_{ \{ \wt{S}_1^2 > \wt{S}_2^2 \}}  
\right)^2  \\
& \le E \left( 
\sum_{j=1}^5 \left| \frac{ \partial [\gamma^\le(\theta) \mu_1 ] }{ \partial \theta_j } \right|  | \wh{\theta}_{Nj} - \theta_j | \eins_{ \{ \wt{S}_1^2 > \wt{S}_2^2 \}} 
\right)^2 \\
&\le C^2 E \left( 5 \max_j | \wh{\theta}_{Nj} - \theta_j |^2 \eins_{ \{ \wt{S}_1^2 > \wt{S}_2^2 \}} \right)^2 \\
& \le      25 C^2 \sqrt{ E \sum_{j=1}^5 | \wh{\theta}_{Nj} - \theta_j |^4 } \sqrt{ P( \wt{S}_1^2 > \wt{S}_2^2 ) }.
\end{align*}
Using
\[
E \sum_{j=1}^5 | \wh{\theta}_{Nj} - \theta_j |^4 = O(1/N^2)
\]
and $ P( \wt{S}_1^2 > \wt{S}_2^2 ) = O( 1/N^2 ) $, see (\ref{HigherOrderEstimateSampleVariance}) and (\ref{EstThetaHatSq}), we therefore arrive at
\[
N E( U_N )^2 =  O( 1/N ).
\]
If, contrary, $ \sigma_1^2 > \sigma_2^2 $, we observe that $ J = \eins_{ \{ \wt{S}_1^2 \le \wt{S}_2^2 \}} $, and similar arguments combined with (\ref{HigherOrderEstimateSampleVariance2}) imply
\[
N E( U_N )^2 = O( 1/N ).
\]  
Putting things together, we therefore arrive at assertion (i),
\begin{align*}
\wh{\mu}_N(\gamma) - \mu 
& = \biggl[ (\nabla [ \gamma^\le(\theta) \mu] + \nabla [(1-\gamma^\le(\theta)) \mu ]) \eins_{\{ \sigma_1^2 < \sigma_2^2 \}} \\
& \quad +
(\nabla [\gamma^>(\theta) \mu] + \nabla [(1-\gamma^>(\theta)) \mu]) \eins_{\{ \sigma_1^2 > \sigma_2^2 \}} \biggr]( \wh{\theta}_N - \theta) + R_N^\gamma,
\end{align*}
since $ \mu_1 = \mu_2 = \mu $ by assumption.

It remains to discuss the form of the leading term when $ \mu = \mu_1 = \mu_2 $ by calculating the partial derivatives. Observe that
\[
\frac{ \partial [ \gamma^\le(\theta) \mu_1 + (1-\gamma^\le(\theta)) \mu_2 ] }{ \partial \sigma_i^2 } = \frac{ \partial \gamma^\le( \theta ) }{ \partial \sigma_i^2 } \mu_1
- \frac{ \partial \gamma^\le( \theta ) }{ \partial \sigma_i^2 } \mu_2
\]
and thus vanishes if $ \mu_1 = \mu_2 $. Hence, the asymptotic linear statistic governing $ \wh{\mu}_N(\gamma) - \mu $ does not depend on the sample variances. Further,
\[
\frac{ \partial [\gamma^\le(\theta) \mu_1 + (1-\gamma^\le(\theta)) \mu_2 ] }{ \partial \mu_1 } = \frac{ \partial \gamma^\le(\theta)}{ \partial \mu_1 } \mu_1 + \gamma^\le( \theta ) - \frac{ \partial \gamma^\le(\theta) }{ \partial \mu_1 } \mu_2
\]
and
\[
\frac{ \partial [\gamma^\le(\theta) \mu_1 + (1-\gamma^\le(\theta)) \mu_2 ] }{ \partial \mu_2 } = \frac{ \partial \gamma^\le(\theta)}{ \partial \mu_2 } \mu_1  - \frac{ \partial \gamma^\le(\theta) }{ \partial \mu_2 } \mu_2 - \gamma^\le(\theta) +1.
\]
Hence, under the common mean constraint $ \mu_1 = \mu_2 $ those two partial derivatices are given by $ \gamma^\le(\theta) $ and $ 1-\gamma^\le(\theta) $, respectively. Thus, we may conclude that 
\[
\wh{\mu}_N(\gamma) - \mu = \gamma(\theta) ( \overline{X}_1 - \mu ) + (1-\gamma(\theta)) ( \overline{X}_2 - \mu ) + R_N,
\]
with $ N E( R_N^2 ) = O(1/N) $, which verifies (ii).

(iii)  Assertion (ii) shows that $ \wh{\mu}_N(\gamma) $ is an asymptotically linear two-sample statistic with kernels $ h_1(x) = \gamma(\theta) (x-\mu_1) $ and $ h_2(x) = (1-\gamma(\theta)) (x-\mu_2) $, which 
satisfy $ \int | h_i(x) |^4 \, d F_i(x) < \infty $, $ i = 1, 2 $,
and  a remainder term $ R_N $ with $ N E(R_N^2) = o(1) $, as $ N \to \infty $,
such that (\ref{CrucialCondRemainder1}) holds. It remains to verify (\ref{CrucialCondRemainder}).  
Our starting point is the Taylor expansion (\ref{TaylorExpansion}), where we have to
study the	second sum. First observe that replacing $ \partial_{ij}[\gamma^{\le}(\theta)\mu_1](\wh{\theta}_N^*) $ by 
$ \partial_{ij}[\gamma^{\le}(\theta)\mu_1](\theta) $ gives error terms $ F_{N,ij}$ satisfying
\begin{align*}
E( F_{N,ij}^2 ) & = E[ \{ \partial_{ij}[\gamma^{\le}(\theta)\mu_1](\wh{\theta}_N^*) -
\partial_{ij}[\gamma^{\le}(\theta)\mu_1](\theta) \} (\wh{\theta}_{Ni} - \theta_i )(\wh{\theta}_{Nj}-\theta_j) ]^2 \\
& \le 
\sqrt{ E \left( \partial_{ij}[\gamma^{\le}(\theta) \mu_1 ] \biggr|_{z=\theta}^{z=\wh{\theta}_N^*} \right)^4 } 
\sqrt{ E( \wh{\theta}_{Ni} - \theta_i )^4 ( \wh{\theta}_{Nj} - \theta_j )^4 }.
\end{align*}	
The first factor is $ o(1) $, by Fubini's theorem, since the integrand is continuous, bounded and converges to $0$, as $ N \to \infty $, a.s. Further
\[
E( ( \wh{\theta}_{Ni} - \theta_i )^4 ( \wh{\theta}_{Nj} - \theta_j )^4  )
\le \sqrt{ E( \wh{\theta}_{Ni} - \theta_i )^8 } \sqrt{ E( \wh{\theta}_{Ni} - \theta_i )^8 }
= O(N^{-4}),
\]
such that 
\begin{equation} 
\label{FN_rateij}
N^2 E( F_{N,ij}^2 ) = o(1) N^2 O(N^{-2}) = o(1),
\end{equation}
as $ N \to \infty $, follows. In what follows, denote by $ \wt{R}_N $  the
second term in (\ref{TaylorExpansion}) with $ \wh{\theta}_N^* $ replaced by $ \theta $,
i.e.
\[
\wt{R}_N = 
\frac12 \sum_{i,j=1}^5 \partial_{ij} [\gamma^\le(\theta) \mu_1]( \theta ) ( \wh{\theta}_{Ni} - \theta_i )( \wh{\theta}_{Nj} - \theta_j ),
\]
such that $ F_N = \frac{1}{2} \sum_{i,j=1}^5 F_{N,ij} $ satisfies
\[ R_N = \wt{R}_N + F_N. \] 
Clearly, (\ref{FN_rateij}) implies
\begin{equation} 
\label{FN_rate}
N^2 E( F_{N}^2 ) = o(1),
\end{equation}
In order to show
\[
N^2 E( R_N - R_{N-1})^2 = o(1),
\]
as $ N \to \infty $, by virtue of (\ref{FN_rate}) it suffices to show that
\begin{equation}
\label{ToShowTilde}
N^2 E(\wt{R}_N - \wt{R}_{N-1})^2 = o(1),
\end{equation}
as $ N \to \infty $. Indeed, since we already know that $ N^2 E(R_N^2) = o(1) $, we have 
\begin{align*}
N^2 E(\wt{R}_N^2) & = N^2 E(R_N^2) + N^2 E( F_N^2 ) - 2 N^2 E( R_N F_N )
\end{align*}
with $ N^2 E( F_N^2 ) = o(1) $ and
\[
| N^2 E (R_N F_N) | \le \sqrt{ N^2 E( R_N^2 ) } \sqrt{ N^2 E( F_N^2 ) } = o(1),
\]
as $ N \to \infty $. Hence, $ N^2 E( \wt{R}_N^2 ) = o(1) $, as $ N \to \infty $, as well.
Due to the high rate of convergence of $ F_N $ in (\ref{FN_rate}), it follows that $ N^2 E( F_N - F_{N-1} )^2 = o(1) $, as $ N \to \infty $, since
\begin{align*}
N^2   E( F_N - F_{N-1} )^2 & \le N^2 \left( \sqrt{ E( F_N^2 ) } + \sqrt{ E( F_{N-1}^2 ) } \right)^2 \\
& = N^2 E( F_N^2 ) + N^2 E( F_{N-1}^2 ) + 2 \sqrt{ N^2 E( F_N^2 ) } \sqrt{ N^2 E( F_{N-1}^2 ) } \\
& = o(1),
\end{align*}
as $ N \to \infty $. Now it follows that (\ref{ToShowTilde}) implies $ N^2E(R_N - R_{N-1})^2 = o(1) $, as $ N \to \infty $, since by Minkowski's inequality
\begin{align*}
N^2 E( R_N - R_{N-1} )^2 &= N^2 E( \wt{R}_N - \wt{R}_{N-1} + (F_N - F_{N-1}) )^2 \\
&=  N^2 E( \wt{R}_N - \wt{R}_{N-1} )^2 +
N^2 E( F_N - F_{N-1} )^2 \\
& \quad + 2 N^2 E( \wt{R}_N - \wt{R}_{N-1} )( F_N - F_{N-1} ),
\end{align*}
where
\begin{align*}
| N^2 E( \wt{R}_N - \wt{R}_{N-1} )( F_N - F_{N-1} ) |
& \le N^2 \sqrt{ E( \wt{R}_N - \wt{R}_{N-1} )^2 } \sqrt{ E( F_N - F_{N-1} )^2 } \\
& \le \sqrt{ N^2 E( \wt{R}_N - \wt{R}_{N-1} )^2 } \sqrt{ N^2 E( F_N - F_{N-1})^2 } \\
& = o(1),
\end{align*}
as $ N \to \infty $. This verifies that it is sufficient to show (\ref{ToShowTilde}).
As the sum is finite and the second order derivatives $ \partial_{ij}[ \gamma{\le}( \theta ) \mu_1 ]( \theta ) $ do not depend on $N$ and are bounded, (\ref{ToShowTilde}) follows, if we establish,
for each pair $ i, j \in \{ 1, \dots, 5 \} $,
\begin{equation}
\label{ToShowDiffEstimates}
N^2 E[(\wh{\theta}_{Ni} - \theta_i ) (\wh{\theta}_{Nj} - \theta_j ) -
(\wh{\theta}_{N-1,i} - \theta_i ) (\wh{\theta}_{N-1,j} - \theta_j ) ]^2 = o(1),
\end{equation}
as $ N \to \infty $. We shall make use of the decomposition $ \wh{\theta}_{Ni} - \theta_i  = L_{Ni} + R_{Ni} $ in a linear statistic and an additional remainder term (if $ \wh{\theta}_{Ni} $ is not a sample mean), for $ i = 1, \dots, 5 $. (\ref{ToShowDiffEstimates}) follows, if we prove
\begin{align}
\label{F1}
N^2 E( L_{Ni} L_{Nj} - L_{N-1,i} L_{N-1,j} )^2 = o(1), \\
\label{F2}
N^2 E( L_{Ni} R_{Nj} - L_{N-1,i} R_{N-1,j} )^2 = o(1), \\
\label{F3}
N^2 E( R_{Ni} R_{Nj} - R_{N-1,i} R_{N-1,j} )^2 = o(1),
\end{align}
as  $ N \to \infty $, for $ i, j = 1, \dots, 5 $.
Those estimates will follow from the following estimates for $ L_{Ni} $ and $ R_{Ni} $:
\begin{align}
\label{L1}
E( L_{Ni}^4 ) &= O(1/N^2), \\
\label{L2}
E( L_{Ni} - L_{N-1,i} )^4 &= O(1/N^4), \\
\label{R1}
E( R_{Ni}^4 ) &= O(1/N^4), \\
\label{R2}
E(R_{Ni} - R_{N-1,i})^2 &= O(1/N^3), \\
\label{R3}
E( R_{Ni} - R_{N-1,i} )^4 &= O( 1/N^6 ),
\end{align}
as $ N \to \infty $, for $ i, j = 1, \dots, 5 $. 
(\ref{L2}) holds, because for a linear statistic, say, 
$ L_N = N^{-1} \sum_{i=1}^N \xi_i $, with i.i.d. mean zero summands $ \xi_1, \dots, \xi_N $ with finite fourth moment, the formula
\[
L_{N} - L_{N-1} = \left( \frac1N - \frac{1}{N-1} \right) \sum_{i<N} \xi_i + \frac{\xi_N}{N}
\]
leads to the estimate
\[
\| L_{N} - L_{N-1} \|_{L_4} \le \frac{1}{N(N-1)} \left\| \sum_{i<N} \xi_i \right\|_{L_4} + \frac{ ( E( \xi_1^4 ) )^{1/4} }{ N } = O( 1/N ),
\]
since $ E( \sum_{i<N} \xi_i )^4 = O(N^2) $. For the sample means the remainder term vanishes, such that (\ref{R1})-(\ref{R3}) hold trivially, and for the sample variances (\ref{R1})-(\ref{R3}) are shown in Lemma~\ref{SampleVarAsLinear}. Now (\ref{F1})-(\ref{F3}) can be shown as follows. For each pair $ i, j \in \{ 1, \dots, 5 \} $ the  identity
\[ 
L_{Ni} R_{Nj} - L_{N-1,i} R_{N-1,j} 
= (L_{Ni} - L_{N-1,i}) R_{Nj} + L_{N-1,i}(R_{Nj} - R_{N-1,j} ) 
\] 
yields
\begin{align*}
\| L_{Ni} R_{Nj} - L_{N-1,i} R_{N-1,j} \|_{L_2} 
&\le \| L_{N,i} - L_{N-1,i} \|_{L_4} \| R_{Nj} \|_{L_4}   \\
&+ \| L_{N-1,i} \|_{L_4} \| R_{Nj} - R_{N-1,i} \|_{L_4} \\
& = O(1/N) O(1/N) + O(1/\sqrt{N}) O( 1 / N^{3/2} ) \\ 
& = O(1 / N^2 ),
\end{align*}
such that
\[
E( L_{Ni}R_{Nj} - L_{N-1,i} R_{N-1,j} )^2 = O(1/N^4),
\]
which verifies (\ref{F2}). Similarly, we have \[ L_{Ni} L_{Nj} - L_{N-1,i} L_{N-1,j} = (L_{Ni} - L_{N-1,i}) L_{Nj} + L_{N-1,i}( L_{Nj} - L_{N-1,j} ), \] where
\begin{align*}
E( (L_{Ni} - L_{N-1,i})^2 L_{Nj}^2  ) &\le \sqrt{ E( L_{Ni} - L_{N-1,i} )^4 } \sqrt{ E( L_{Nj}^4 ) } \\
&= O(1/N^2) O(1/N) \\
& = O(1/N^3),
\end{align*}
which leads to (\ref{F1}).
Combining those results shows that the remainder term, $ R_N $, also satisfies
the condition (\ref{CrucialCondRemainder}).  
Therefore, the consistency and asymptotic unbiasedness of the jackknife variance estimator follows from Theorem~\ref{TwoSampleJack}. The proof is complete. $ \hfill \qed$

\bibliographystyle{chicago}
\bibliography{lit}


\end{document}